\documentclass{amsart}
\usepackage{latexsym,amsxtra,amscd,ifthen,amsmath,color, multicol,mathdots,fancyvrb}
\usepackage{amsfonts}
\usepackage{verbatim}
\usepackage{amsmath}
\usepackage{amsthm}
\usepackage{amssymb,changes}
\usepackage{graphicx} \usepackage{caption} \usepackage{subcaption}
\usepackage[notcite,notref, final]{showkeys}
\usepackage[draft]{hyperref}
 \usepackage[all,cmtip]{xy}

\usepackage{enumitem}
\usepackage{tikz}
\usetikzlibrary{arrows.meta}

\numberwithin{equation}{section}

\theoremstyle{plain}
\newtheorem{theorem}{Theorem}[section]
\newtheorem{lemma}[theorem]{Lemma}

\newtheorem{proposition}[theorem]{Proposition}

\theoremstyle{definition}
\newtheorem{definition}[theorem]{Definition}
\newtheorem{example}[theorem]{Example}
\newtheorem{hypothesis}[theorem]{Hypothesis}
\newtheorem{notation}[theorem]{Notation}

\newtheorem{remark}[theorem]{Remark}

\newcommand{\aconst}[3]{a_{#1}^{#2 #3}}
\newcommand{\bconst}[2]{b^{#1 #2}}
\newcommand{\id}{\textup{id}}
\newcommand{\Image}{\textup{Im}}
\newcommand{\mc}{\mathcal}
\newcommand{\kk}{\Bbbk}
\newcommand{\kspan}{\textup{span}_{\kk}}

\makeatletter              
\let\c@equation\c@theorem  
\makeatother

\newcommand{\N}{\mathbb{N}}

\begin{document}

\title[PBW deformations of quadratic monomial algebras]{PBW deformations of \\quadratic monomial algebras}
\author[Z. Cline, A. Estornell, C. Walton, M. Wynne]{Zachary Cline, Andrew Estornell, Chelsea Walton,\\ and Matthew Wynne}

\address{Cline: Department of Mathematics, Temple University, Philadelphia, Pennsylvania 19122, USA}
\email{zcline@temple.edu}

\address{Estornell: Department of Computer Science \& Engineering, Washington University, Saint Louis, Missouri 63110, USA}
\email{aestornell@wustl.edu}

\address{Walton: Department of Mathematics, The University of Illinois at Urbana-Champaign, Urbana, Illinois 61801, USA}
\email{notlaw@illinois.edu}

\address{Wynne: Department of Mathematics, Temple University, Philadelphia, Pennsylvania 19122, USA}
\email{matthew@temple.edu}

\bibliographystyle{abbrv}

\begin{abstract}
A result of Braverman and Gaitsgory from 1996 gives necessary and sufficient conditions for a filtered algebra to be a Poincar\'e-Birkhoff-Witt (PBW) deformation of a Koszul algebra. 
The main theorem in this paper establishes conditions equivalent to the Braverman-Gaitsgory Theorem to efficiently determine PBW deformations of quadratic monomial algebras. 
In particular, a graphical interpretation is presented for this result, and we discuss circumstances under which some of the conditions of this theorem need not be checked. Several examples are also provided.
Finally, with these tools, we show that each quadratic monomial algebra admits a nontrivial PBW deformation.
\end{abstract}



\maketitle


\section{Introduction}

Let $\kk$ denote a field throughout, and note that all algebraic structures in this work will be over $\kk$.
The goal of this article is to study \emph{filtered deformations} of  quadratic monomial algebras in the following sense. Consider the quadratic monomial algebra $A:=T(V)/(R)$, where $V$ is a finite-dimensional $\kk$-vector space with basis $x_1, \ldots, x_n$, and $(R)$ is the two-sided ideal of the tensor algebra $T(V)$ generated by the span of a finite set of simple tensors  $R = \{x_i \otimes x_j\}_{i,j}$ in $V \otimes V$. 
For linear maps $\alpha: R \to V$ and $\beta: R \to \kk$, let 
\[
  P = \{r - \alpha(r) - \beta(r) : r \in R\},
\]
and take $D$ to be the filtered algebra $T(V) / (P)$.
Our aim is to present user-friendly necessary and sufficient conditions for the associated graded algebra of $D$ to be isomorphic to $A$ as graded algebras, or by definition, such conditions so that  $D$ is a {\it Poincar\'e-Birkhoff-Witt (PBW) deformation} of $A$ [Definition~\ref{def:PBW}].

In general, PBW deformations of graded algebras are useful because many ring-theoretic and homological properties are preserved under this construction.
For example, if $A$ is an integral domain, prime ring, or Noetherian ring, then so is any PBW deformation (see, e.g., \cite[Section~1.6]{MR}).
The homological Calabi-Yau property under PBW deformations has also been studied in \cite{BT} and \cite{WZ}. Moreover, for
 a survey of recent results on {\it PBW theorems} and methods, see \cite{SW}.

The primary tool used in this work is a fundamental PBW theorem due to  Braverman and Gaitsgory, \cite[Theorem~4.1]{BG}, restated in Theorem~\ref{thm:BG}, which provides necessary and sufficient conditions  for a certain filtered algebra to be a PBW deformation of a Koszul algebra (see Definition~\ref{def:koszul}).
Since all quadratic monomial algebras are Koszul [Proposition~\ref{prop:quadKoszul}], the Braverman-Gaitsgory Theorem applies for our work here.

For our main result, Theorem~\ref{thm:main}, we introduce scalars $\aconst m i j$ and $\bconst i j$ in $\kk$, determined by the equations
\[
  \alpha(x_i \otimes x_j) = \sum_{m=1}^n \aconst m i j x_m, \quad \quad
    \beta(x_i \otimes x_j) = \bconst i j, \quad \text{ for $x_i \otimes x_j \in R$}.
\]
We then apply the theorem of Braverman and Gaitsgory mentioned above to determine conditions on the scalars $\aconst m i j$ and $\bconst i j$, which we label (I, II, III), that are necessary and sufficient for  $D$ to be a PBW deformation of  $A$ as above.

For instance, we have by work of Berger \cite[Proposition~6.1]{berger}: when $R$ is the span of $x_i \otimes x_i$ for some given $i$, then ${D = T(V)/(P)}$ is a PBW deformation of $A = T(V)/(R)$ if and only if $(\alpha+\beta)(x_i \otimes x_i)$ is a polynomial in $x_i$ (i.e., when $a^{ii}_m = 0$ for all $m \neq i$ and $\aconst i i i, \bconst i i$ are free). 
This fact can be easily recovered with our main result Theorem~\ref{thm:main}-- see Example~\ref{ex:loops}.

To contrast with other results in the literature, note that the monomial algebras that have appeared often in the representation theory of finite-dimensional algebras (also known as {\it zero relation algebras}) are somewhat different than the monomial algebras that we study here, which could be infinite-dimensional. Both algebras arise as quotients of path algebras of a quiver by an ideal of monomial relations;  the quiver is typically acyclic in the former case, whereas the quiver here consists of $n$ loops. In the former case there is a classification of the monomial algebras that are rigid in the sense that they do not admit nontrivial (PBW) deformations, due to work of Cibils \cite[Theorem~3.12]{cibils}. On the other hand, we establish via  the framework of Theorem~\ref{thm:main} that  quadratic monomial algebras in the form of $A$ above are never rigid: see Theorem~\ref{thm:nontrivialdefs}. 

\smallskip
This paper is organized as follows. In Section~\ref{sect:background}, we discuss PBW deformations more precisely and provide the theorem of Braverman and Gaitsgory mentioned above.
In Section~\ref{sect:main}, we discuss further monomial algebras in our context, attach a  graph $\Gamma(A)$ to these algebras which depicts their relation space, and we also prove our main result, Theorem~\ref{thm:main}.
The graph $\Gamma(A)$ is used in subsequent sections to help execute Theorem~\ref{thm:main}.
Section~\ref{sect:cases} is devoted to examining special cases when conditions  (II) and (III) of Theorem~\ref{thm:main} need not be checked (after assuming condition (I) holds).
We apply Theorem~\ref{thm:main} to several examples in Section~\ref{sect:examples} to illustrate its usefulness.
Finally, in Section~\ref{sect:nontrivialdefs}, we use Theorem~\ref{thm:main} to show that each quadratic monomial algebra admits a nontrivial PBW deformation.


\section{Background material} \label{sect:background}


In this section, we first recall facts about  graded and filtered $\kk$-algebras, 
the  associated graded algebra corresponding to a filtered algebra $D$, along with   PBW deformations of graded algebras.
We end the section by recalling a theorem of Braverman and Gaitsgory \cite{BG} [Theorem~\ref{thm:BG}] that provides necessary and sufficient conditions for a quadratic (filtered) algebra to be a PBW deformation of a  Koszul algebra [Definition~\ref{def:koszul}]. The Braverman-Gaitsgory Theorem  will be used in subsequent sections to study PBW deformations of quadratic monomial algebras as such algebras are Koszul  (\cite[Corollary~2.4.3]{PP}, restated in Proposition~\ref{prop:quadKoszul}).

\subsection{Graded and filtered algebras} 
  We begin with a discussion of graded and filtered $\kk$-algebras.
 An \emph{$\N$-graded  $\kk$-algebra} $A$ is an associative $\kk$-algebra together with a decomposition of $A$ into $\kk$-subspaces
 $
      A = \bigoplus_{i \geq 0} A_i
    $
    so that $A_i \cdot A_j \subseteq A_{i+j}$ for any $i,j \in \N$. The subspace $A_i$ is referred to as the \emph{degree $i$ part} of $A$.

For example, let $V$ be a vector space over $\kk$ and consider $\kk$-vector spaces $T_0 := \kk$ and 
 $T_i := V^{\otimes i}$ for $i > 0$.
   Then, the {\it tensor algebra} $T(V)$ on $V$, which is  the $\kk$-vector space $\bigoplus_{i=0}^\infty T_i$ with 
    multiplication given by concatenation, is an $\mathbb{N}$-graded $\kk$-algebra.
    If $V$ is $n$-dimensional, say with basis $x_1, \ldots, x_n$, then $T(V)$ is isomorphic to the free algebra $\kk\langle x_1, \ldots, x_n \rangle$. Moreover, let $R$ be a subset of $T(V)$.
    If every element of $R$ is homogeneous of the same degree, that is, $R \subseteq T(V)_i$ for some $i$, then $T(V) / (R)$ will inherit the grading from $T(V)$ above.
  
\medskip 

\noindent {\bf Convention.}  Henceforth, we will often suppress $\otimes$ when expressing elements of the tensor algebra $T(V)$ or quotients thereof.

\medskip

An \emph{$\N$-filtered $\kk$-algebra} $D$ ({\it with increasing filtration}) is an associative $\kk$-algebra together with an increasing sequence of subspaces of $D$
    \[
      \{0\} \subseteq \mc F^0(D) \subseteq \mc F^1(D) \subseteq \cdots \subseteq \mc F^i(D) \subseteq \cdots \subseteq \bigcup\limits_{i=0}^{\infty} \mc F^i(D) = D
    \]
so that $\mc F^i(D) \cdot \mc F^j(D) \subseteq \mc F^{i+j}(D)$ for any $i,j \in \N$.
Moreover, if $D = \bigcup_{i=0}^\infty \mc F^i(D)$ is a filtered algebra, then the \emph{associated graded algebra}, $\text{gr}_{\mc F}(D)$, is the graded algebra
    \[
      \text{gr}_{\mc F}(D) = \bigoplus_{i=0}^\infty \mc F^i(D) / \mc F^{i-1}(D),
    \] 
    where we take $\mc F^{-1}(D) = \{0\}$ by convention.
    The multiplication is given on cosets by 
      [$x$] [$y$] = [$xy$]
    for $x \in \mc F^i(D)$ and $y \in  \mc F^j(D)$, and is
 well-defined.

In this work, the filtered algebra $D$ will be a quotient of the tensor algebra $T(V)$, with filtration inherited from $T(V)$, that is, $\mathcal{F}^j(T(V)) = \bigoplus_{i=0}^j V^{\otimes i}$. So, the filtration $\mathcal{F}$ on $D$ is understood as such.

  \begin{definition} \label{def:PBW}
    Given a graded algebra $A$ and a  filtered algebra $D$ with filtration $\mc F$, we say that $(D, \mc{F})$ is a \emph{Poincar\'e-Birkhoff-Witt (PBW)  deformation} of a graded algebra $A$ if $\text{gr}_{\mc F}(D) \cong A$ as graded algebras.
  \end{definition}


\subsection{The Braverman-Gaitsgory Theorem}
Here, we discuss a result of Braverman and Gaitsgory \cite{BG} which enables us to compute the PBW deformations of certain quadratic algebras, namely of Koszul algebras which are defined below.

\begin{definition} \label{def:koszul}
Let $A = \bigoplus_{i \geq 0} A_i$ be an $\mathbb{N}$-graded
$\kk$-algebra with $A_0 = \kk$.

\begin{enumerate}
\item The left $A$-module $A/\bigoplus_{i \geq 1} A_i \cong \kk$ is referred to as the {\it trivial $A$-module}.

\item Let $M = \bigoplus_{n \in \mathbb{Z}} M_n$ be a $\mathbb{Z}$-graded $A$-module so that $A_i \cdot M_n \subset M_{i+n}$ for all $i, n$. Take $\ell \in
\mathbb{Z}$. Then the {\it shift} of $M$ by $\ell$ is an $A$-module $M(\ell)$ whose degree $n$ part is $M(\ell)_n = M_{n+\ell}$.

\item We say that $A$ is {\it Koszul} if there exists a long exact sequence
$$ \cdots \longrightarrow A(-3)^{\oplus b_3} \longrightarrow A(-2)^{\oplus b_2} \longrightarrow A(-1)^{\oplus b_1}
\longrightarrow A \longrightarrow \kk \longrightarrow 0,$$
where $A(-m)$ is the shift of the regular $A$-module $A$ by the integer $-m$, and  each $b_j$ is some non-negative integer.
\end{enumerate}
\end{definition}

In the setting of part (c) above, we also say that the trivial $A$-module $\kk$ admits a linear minimal graded resolution by free $A$-modules. See notes of Kr\"{a}hmer \cite{Kr} for more information. Moreover, Koszul algebras are necessarily graded and quadratic (see, e.g., \cite[Section~2.1]{PP}), and a class of examples of such graded algebras are given as follows.

\begin{proposition} \label{prop:quadKoszul} \cite[Corollary~2.4.3]{PP}
 A quadratic monomial algebra is Koszul. \qed
\end{proposition}

Now we present the Braverman-Gaitsgory Theorem.

\begin{theorem}[{\cite[Theorems~0.5 and~4.1]{BG}}] \label{thm:BG}
  Let $A = T(V) / (R)$ be a (quadratic) Koszul algebra.
  Let $\alpha: R \to V$ and $\beta: R \to \kk$ be $\kk$-linear maps, and define
  \[
    P = \{r - \alpha(r) - \beta(r) : r \in R\}.
  \]
  Then the filtered algebra $D = T(V) / (P)$ is a PBW deformation of $A$ if and only if the following conditions are satisfied:
  \begin{enumerate}[label=\textup{(\arabic*)}]
    \item
      $P \cap \mc F^1(T(V)) = 0$,
    \item
      $\Image(\alpha ~ \id_V - \id_V ~ \alpha) \subseteq R$,
    \item
      $\alpha(\alpha ~ \id_V - \id_V ~ \alpha) = -(\beta ~ \id_V - \id_V ~ \beta)$,
    \item
      $\beta(\alpha ~ \id_V - \id_V ~ \alpha) \equiv 0$.
  \end{enumerate}
  Here, the maps $\alpha ~ \id_V - \id_V ~ \alpha$ and $\beta ~ \id_V - \id_V ~ \beta$ are defined on the $\kk$-vector space $(R \otimes V) \cap (V \otimes R)$. \qed
\end{theorem}


\section{Main Result} \label{sect:main}

The goal of this section is to provide a method to  efficiently compute all PBW deformations of any given quadratic monomial algebra $A$. We do so by producing Theorem \ref{thm:main}, which provides conditions equivalent to Theorem \ref{thm:BG} when applied to quadratic monomial algebras. 

For a preview of the notation used below, recall that the relation space $R$ of $A$ is given by only quadratic monomials, and our main result can be interpreted easily via a graph $\Gamma(A)$ consisting of vertices and arrows corresponding respectively to generators and relations of $A$. In particular, each path of length three in $\Gamma(A)$ corresponds to an element of a basis $Q$ of the space $(R \otimes V) \cap (V\otimes R)=:RV \cap VR$. Since linear maps on this space determine (the nontrivial) restrictions for a filtered algebra to be a PBW deformation of $A$ (see Theorem~\ref{thm:BG}), we will be able to compute all PBW deformations of $A$ by analyzing length three paths of the graph $\Gamma(A)$.


\subsection{Quadratic monomial algebras} \label{subsect:quadmon}
\begin{hypothesis} \label{hyp:A}($A$).
  From now on,  $A$ will denote a finitely generated quadratic monomial algebra generated in degree one.
  That is, we have 
  \[
    A = \kk \langle x_1, x_2, \ldots, x_n \rangle / (R) ~\cong~ T(V) / (R),
  \]
  where $\deg(x_i) = 1$ for all $i$, $R$ is the span of a finite set of monomials in the $x_i$ of  degree 2, and  $V$ is the $\kk$-vector space with basis $\{ x_1, x_2, \ldots, x_n \}$.
\end{hypothesis}
  
  \begin{notation}[$\Gamma(A)$] \label{notation:gamma(A)}
    For each such algebra $A$, we can associate to it a directed graph $\Gamma(A)$.
    The vertices of $\Gamma(A)$ will be $\{1, 2, \dots, n\}$  corresponding to the chosen basis  $\{x_1, x_2, \ldots, x_n\}$ of $V$.
    For each pair of vertices, $i$ and $j$, let there be a single arrow from $i$ to $j$ if and only if $x_i x_j$ is a monomial in the relation space ($R$).
  \end{notation}
  
    The graph $\Gamma(A)$ defined above is the complement of the \textit{Ufnarovskii graph} $U(A)$ of $A$ defined in \cite{Uf} in the sense that  $\Gamma(A)$ depicts the basis of the relation space of $A$, while the graph $U(A)$ depicts the monomial basis (of nonzero elements) of $A$.
  
  \begin{example} \label{ex:3gens}
    Consider the quadratic monomial algebra
    \[
      A = \kk \langle x_1, x_2, x_3 \rangle / (x_1^2, ~ x_1 x_2, ~ x_2 x_3, ~ x_3^2 ).
    \]
    In this case, we have $\Gamma(A)$ is the following graph:
    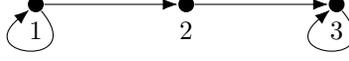
\begin{figure}[h] \label{fig:ex1}
    \begin{center}
      \begin{tikzpicture}
\tikzset{vertex/.style = {shape = circle,fill = black,minimum size = 6pt,inner sep=0pt}}
        \tikzset{edge/.style = {-{Latex[length=2mm]}}}
        \tikzset{loopedge/.style = {-{Latex[length=2mm]}, in=45, out=135, loop}}
        \tikzset{every loop/.style={-{Latex[length=2mm]}, min distance=10mm,in=220,out=310,looseness=1}}
      
        \node[vertex, label=below:$1$] (a) at (0,0) {};
        \node[vertex, label=below:$2$] (b) at (2,0) {};
        \node[vertex, label=below:$3$] (c) at (4,0) {};
        \node[vertex, white] (w) at (2,0.5) {};
        
        \draw[loopedge] (a) to (a);
        \draw[edge] (a) to (b);
        \draw[edge] (b) to (c);
        \draw[loopedge] (c) to (c);
      \end{tikzpicture}
    \end{center}
    \caption{$\Gamma(A)$ for $A = \kk \langle x_1, x_2, x_3 \rangle / (x_1^2, ~ x_1 x_2, ~ x_2 x_3, ~ x_3^2 )$}
   \end{figure}
  \end{example}

 Now since all quadratic monomial algebras are Koszul algebras [Proposition~\ref{prop:quadKoszul}] the Braverman-Gaitsgory Theorem [Theorem~\ref{thm:BG}] applies to $A$ above.
  So, given linear maps $\alpha: R \to V$ and $\beta: R \to \kk$, we define 
  \[
    P = \{ r - \alpha(r) - \beta(r) : r \in R \} \subseteq R \oplus V \oplus \kk,
  \]
  and we let $D = T(V) / (P)$ denote the corresponding filtered algebra. Here, $V = \bigoplus_{m=1}^n \kk x_m$.
  Since $R$ has a basis consisting of monomials of  degree 2, $\alpha$ and $\beta$ are determined by the following scalars:
  \[
    \{ \aconst m i j \in \kk: x_i x_j \in R , \ m=1, 2, \dots, n \} \quad \text{and} \quad
    \{ \bconst i j \in \kk: x_i x_j \in R \},
  \]
  which are respectively defined by 
  \[ 
    \alpha(x_i x_j) = \sum_{m = 1}^n \aconst m i j x_m \quad \text{and} \quad
    \beta(x_i x_j) = \bconst i j.
  \]
  
  \begin{definition}[$\aconst m i j$, $\bconst i j$] \label{def:defparams}
    Retain the notation above.
    \begin{enumerate}
   \item We refer to the scalars $\aconst m i j$ and $\bconst i j$ as the \emph{filtration parameters} for the filtered algebra $D= T(V) / (P)$ associated to the graded algebra $A$ in Hypothesis~\ref{hyp:A}. 
 \item   If, further, $D$ is a $PBW$ deformation of $A$, then we refer to $\aconst m i j$ and $\bconst i j$  as  \emph{deformation parameters} of $D$ which deform $A$.
    \end{enumerate}
  \end{definition}

  In the next section, we use Theorem~\ref{thm:BG} to determine explicit necessary and sufficient conditions on the filtration parameters of $D$ so that 
they are deformation parameters of $D$.


\subsection{Braverman-Gaitsgory Theorem for quadratic monomial algebras} 
Recall the role of the space $RV \cap VR$ in Theorem~\ref{thm:BG}.
The following lemma gives a simple way of producing a basis of $RV \cap VR$. This basis, which we will denote by $Q$, represents all paths of length three in the graph $\Gamma(A)$ from Notation~\ref{notation:gamma(A)}.

  \begin{lemma}[$Q$] \label{lem:RVVRbasis}
    Let $Q := \{ x_i  x_j  x_k : x_i  x_j, \ x_j  x_k \in R \}$.
    Then $Q$ is a basis for the space $RV \cap VR$. As a consequence, $x_ix_jx_k \in Q$ if and only if $i\rightarrow j\rightarrow k$ is a path contained in $\Gamma(A)$; note that $i,j,k$ need not be distinct. 
  \end{lemma}

    \begin{proof}
      It is clear that $Q \subseteq RV \cap VR$.
      Now suppose $R = \kspan (\{x_{i_1} x_{i_2} \}_{i=1}^t)$ and take an arbitrary element of $RV \cap VR$,
      \[
        \sum_{i=1}^t x_{i_1}  x_{i_2}  v_i 
        = \sum_{i=1}^t w_i x_{i_1}  x_{i_2},
      \]
      with $v_i, w_i \in V$.
      Write $v_i = \sum_{s = 1}^n \lambda_{i,s} x_s$ and $w_i = \sum_{s = 1}^n \mu_{i,s} x_s$, for $\lambda_{i,s}, \mu_{i,s} \in \kk$.
      Then we have
      \[
        \sum_{s=1}^n \sum_{i=1}^t \lambda_{i,s} x_{i_1}  x_{i_2}  x_s 
        = \sum_{s=1}^n \sum_{i=1}^t \mu_{i,s} x_s  x_{i_1} x_{i_2}.
      \]
      Note that $\{x_a  x_b  x_c\}_{1 \leq a,b,c \leq n}$ is a basis for $V^{\otimes 3}$.
      We conclude that 
      if the coefficient of some $x_a x_b  x_c$ is nonzero, then we must have $x_a x_b \in R$ and $x_b  x_c \in R$.
      Thus, the element $x_a  x_b  x_c$ is in $Q$, so $Q$ spans $RV \cap  VR$.
      The linear independence of $Q$ follows from the basis of $V^{\otimes 3}$ given above.
    \end{proof}
  
The following theorem provides  conditions equivalent to conditions (1)-(4) of Theorem \ref{thm:BG} when applied to quadratic monomial algebras. Recall that the maps $\alpha$ and $\beta$ are given by the filtration parameters $\aconst * * *$ and $\bconst * *$ respectively. Subject to certain restrictions, these filtration parameters are deformation parameters which determine the relation space $P$ such that $D = T(V) / (P)$ is a PBW deformation of $A = T(V) / (R)$. These restrictions can be defined exclusively on elements of $Q$. Moreover, elements of the relation space $R$ (which determines $Q$) can be interpreted as arrows of the graph $\Gamma(A)$. As a result, the conditions in Theorem \ref{thm:main} can be visualized using $\Gamma(A)$ as we will see below.
  
  \begin{theorem}[$d_r^{ijk}$] \label{thm:main}
    With $A$, $R$, $P$, and $D$ as in Section~\ref{subsect:quadmon} and $Q$ as in Lemma~\ref{lem:RVVRbasis}, the filtered algebra $D$ is a PBW deformation of the quadratic monomial algebra $A$ 
    (i.e. the filtration parameters $\aconst m i j$ and $\bconst i j$ are deformation parameters) if and only if the following conditions hold for each element $x_i x_j x_k$ of $Q$.
    \begin{enumerate}[label=\textup{(\Roman*)}]
      \item We have the conditions below (as illustrated in Figure~2):
        \begin{enumerate}[label=\textup{(\alph*)}]
           \item
            $\aconst m i j = 0$ for every $m \neq i$ such that $x_m x_k \not \in R$,
            
            \smallskip
            
          \item
            $\aconst {m'} j k = 0$ for every $m' \neq k$ such that $x_i x_{m'} \not \in R$,
            
             \smallskip
             
          \item
            $\aconst i i j = \aconst k j k$ if $x_i x_k \not \in R$.
        \end{enumerate}
        \medskip
      \item
        For each $r \in \{1, \ldots, n\}$, let
        \[
          d_r^{ijk} := \sum_{ \substack{m \in \{1, \ldots, n\} \\ x_m x_k \in R}} \aconst m i j \aconst r m k \ - \sum_{\substack{m' \in \{1, \ldots, n\} \\ x_i x_{m'} \in R}} \aconst r i {m'} \aconst {m'} j k .
        \]
        \begin{enumerate}[label = \textup{(\alph*)}]
          \item
            If $i \not = k$, then $d_r^{ijk} = 
            \begin{cases}
              0, & \text{ if } r \neq i, k \\
              - \bconst i j, & \text{ if } r = k \\
              \bconst j k, & \text{ if } r = i .
            \end{cases}$
          \item
             If $i = k$, then $d_r^{ijk} = 
            \begin{cases}
              0, & \text{ if } r \neq i \\
              \bconst j i - \bconst i j, & \text{ if } r = i.
            \end{cases} $
        \end{enumerate}
        
      \item
        Moreover,	
        \begin{equation*}
          \sum_{ \substack{m \in \{1, \ldots, n\} \\ x_m x_k \in R}} \aconst m i j \bconst m k  \ - \sum_{\substack{m' \in \{1, \ldots, n\} \\ x_i x_{m'} \in R}} \bconst i {m'} \aconst {m'} j k = 0.
        \end{equation*}
    \end{enumerate}
  \end{theorem}

 \begin{figure}[h]
    \label{fig:visualization}
    \begin{center}
      \begin{tikzpicture}
        \tikzset{vertex/.style = {shape = circle,fill = black,minimum size = 6pt,inner sep=0pt}}
        \tikzset{edge/.style = {-{Latex[length=2mm]}}}
        \tikzset{loopedge/.style = {-{Latex[length=2mm]}, in=45, out=135, loop}}
        \tikzset{every loop/.style={-{Latex[length=2mm]}, min distance=20mm,in=220,out=310,looseness=1}}
      
        \node[vertex, label=below:$i$] (a) at (0,0) {};
        \node[vertex, label=below:$j$] (b) at (3,0) {};
        \node[vertex, label=below:$k$] (c) at (6,0) {};
        \node[vertex, label={$m' \neq j ,k$}] (d) at (1.25,2) {};
        \node[vertex, label={$m \neq i, j$}] (e) at (4.75,2) {};
        
        \draw[edge] (a) to (b);
        \draw[edge] (b) to (c);
        \draw[edge, bend right, dashed] (a) to node[below]{ (c): $\aconst i i j = \aconst k j k$ }(c) ;
        \draw[edge, dashed] (a) to node[left]{ (b): $ \aconst {m'} j k = 0 ~$} (d);
        \draw[edge, dashed] (e) to node[right]{ ~(a): $ \aconst m i j = 0 $} (c);
      \end{tikzpicture}
    \end{center}
    \caption{Depiction of condition (I) of Theorem~\ref{thm:main} in terms of the graph $\Gamma(A)$, with $*\dashrightarrow\circ$ denoting $x_*x_\circ \not\in R$.}
  \end{figure}
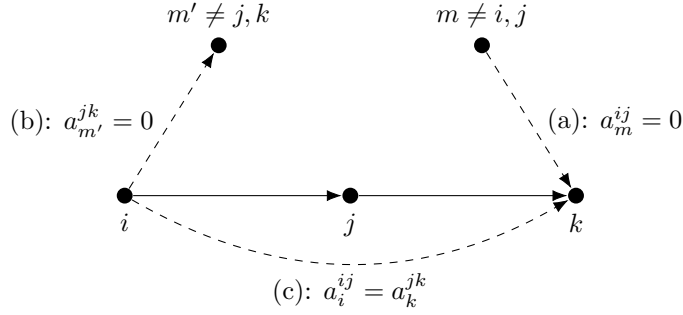

  \begin{proof}
    We will show first that condition (1) of Theorem~\ref{thm:BG} holds for all quadratic monomial algebras. 
    We will then proceed to prove that condition (2) of Theorem~\ref{thm:BG} is equivalent to (I) in the case of quadratic monomial algebras. Finally, under condition (I), we prove that conditions (3) and (4) of Theorem~\ref{thm:BG} are equivalent to conditions (II) and (III), respectively.
    
    Recall that condition (1) of Theorem~\ref{thm:BG} requires that
    \[
      P \cap \mc F^1(T(V)) = {0}. \\
    \]
    Notice that $P$ has a basis consisting of degree 2 terms, each with a single distinct leading degree 2 monomial. Therefore, every element of $P$ can be expressed as
\begin{gather*}
\sum_{ \substack {x_i x_j \in R}} c_{ij} (x_i x_j -\aconst 1 i j x_1 - \dots -\aconst n i j x_n - \bconst i j)
\end{gather*}
for some constants $c_{ij} \in \kk$. Now the degree of each element of $P$ is determined by the values of $c_{ij}$. Namely, if any $c_{ij} \neq 0$, then the above element has degree 2, and if each $c_{ij} = 0$, then it has degree 0.
We obtain that each polynomial in $P$ must necessarily be of degree 2 or 0. Therefore any nonzero element of $P$ cannot be in $\mc F^1(T(V))$.
    
    To see that (I) is equivalent to condition (2) of Theorem~\ref{thm:BG}, say that the filtration parameters corresponding to the map $\alpha$ as in Definition~\ref{def:defparams}
     are given by
    \[
      \alpha(x_i x_j) = \sum_{m = 1}^n \aconst m i j x_m \quad \text{and} \quad \alpha(x_j x_k) = \sum_{m' = 1}^n \aconst {m'} j k x_{m'}. 
    \]
    Therefore, for $x_i x_j x_k \in Q$, we have
    \begin{align*}
      (\alpha~\id_V - \id_V~\alpha)(x_i x_j x_k)
        &= \sum_{m = 1}^n \aconst m i j x_mx_k  -  \sum_{m' = 1}^n \aconst{ m'} j kx_i x_{m'} \\
        &= \sum_{m \neq i} \aconst m i j x_m x_k - \sum_{m' \neq k} \aconst {m'} j k x_i x_{m'} + (\aconst i i j - \aconst  k j k) x_i x_k.
    \end{align*}
    Since $Q$ is a basis of $RV \cap VR$, by Lemma \ref{lem:RVVRbasis}, we see that condition (2) of Theorem~\ref{thm:BG}, namely that $\Image(\alpha~\id_V - \id_V~\alpha) \subseteq R $, holds
    if and only if for each $x_i x_j x_k \in Q$ the expansion of the displayed equation above is in $ R $.
    It is now clear that condition~(I) is equivalent to condition (2) of Theorem~\ref{thm:BG}.
    
    Next, to study condition (II), we assume that condition (I) holds and we use the filtration parameters to evaluate the maps $\alpha(\alpha ~ \id_V - \id_V ~ \alpha)$ and $ -(\beta ~ \id_V - \id_V ~ \beta)$ of Theorem~\ref{thm:BG}(3) on $Q$.
    First, by condition (I), we have that for $x_i x_j x_k \in Q$,
  \begin{equation} \label{eq:alpha}
      (\alpha ~ \id_V - \id_V ~ \alpha)(x_i x_j x_k)
        = \sum_{\substack{m \in \{1, \ldots, n\} \\ x_m x_k \in R}} \aconst m i j x_m x_k - \sum_{\substack{m' \in \{1, \ldots, n\} \\ x_i x_{m'} \in R}} \aconst {m'} j k x_i x_{m'}.
  \end{equation}
    Thus, we compute $\alpha(\alpha ~ \id_V - \id_V ~ \alpha)(x_i x_j x_k)$ as follows:
    \begin{align*}
      \alpha & \left( \sum_{\substack{m \in \{1, \ldots, n\} \\ x_m x_k \in R}} \aconst m i j x_m x_k - \sum_{\substack{m' \in \{1, \ldots, n\} \\ x_i x_{m'} \in R}}  \aconst {m'} j k x_i x_{m'}  \right) \\
      & \hspace{.2in} = \sum_{\substack{m \in \{1, \ldots, n\} \\ x_m x_k \in R}} \aconst m i j  \left( \sum_{r = 1}^n \aconst r m k x_r \right) - \sum_{\substack{m' \in \{1, \ldots, n\} \\ x_i x_{m'} \in R}} \aconst {m'} j k \left( \sum_{r = 1}^n \aconst r i {m'} x_r \right)
 ~~=~~ \sum_{r = 1}^n 
     d^{ijk}_r x_r .
    \end{align*}
    On the other hand, we have 
    \[
      -(\beta ~ \id_V - \id_V ~ \beta)(x_i x_j x_k) = - \bconst i j x_k + \bconst j k x_i.
    \]
    Therefore, by comparing coefficients, condition (3) of Theorem~\ref{thm:BG} is equivalent to condition (II).
   
    Lastly, we show that condition (4) of Theorem~\ref{thm:BG} is equivalent to condition~(III).
    To conclude this, apply $\beta$ to \eqref{eq:alpha} to obtain:
    \begin{align*}
      \beta \left((\alpha ~ \id_V - \id_V ~ \alpha)(x_i x_j x_k) \right)
      &= \sum_{\substack{m \in \{1, \ldots, n\}  \\ x_m x_k \in R}} \aconst m i j \bconst m k  - \sum_{\substack{m' \in \{1, \ldots, n\} \\ x_i x_{m'} \in R}} \bconst i {m'} \aconst {m'} j k.
    \end{align*}

    \vspace{-.2in}
    
    \end{proof}

\begin{example}
  As a continuation of Example~\ref{ex:3gens}, we compute all PBW deformations of $A = \kk \langle x_1, x_2, x_3 \rangle / (x_1^2,~ x_1 x_2,~ x_2 x_3,~ x_3^2 )$.
   
  To start, note that the basis $Q$ of $RV \cap VR$ consists of the elements $x_1^3$, $x_1^2 x_2$, $x_1 x_2 x_3$, $x_2 x_3^2$, $x_3^3$.
  Applying condition (I) of Theorem \ref{thm:main} to each basis element yields
  \begin{equation} \label{eq:exampleconditioni}
    \aconst 2 1 1 = \aconst 3 1 1 = \aconst 3 1 2 = \aconst 1 2 3 = \aconst 1 3 3 = \aconst 2 3 3 = 0 \quad \text{and} \quad
    \aconst 1 1 2 = \aconst 3 2 3.
  \end{equation} 
  For example, let us apply Theorem~\ref{thm:main}(I) to $x_1^2x_2 \in Q$, where $i=j=1$ and $k=2$.
  This can be visualized by considering Figure~2 as a subgraph of Figure~1 with $i=j=1$ and $k=2$.
 In $\Gamma(A)$  there are no arrows from vertex~2 to vertex~2, or from 3 to 2; so, Theorem~\ref{thm:main}(I.a) with $m = 3,2$ respectively  will yield $\aconst 2 1 1 = \aconst 3 1 1 = 0$.  Similarly, there is no arrow from 1 to 3;  Theorem~\ref{thm:main}(I.b) with $m' = 3$  will yield $\aconst 3 1 2 = 0$. Now \eqref{eq:exampleconditioni} holds by similar uses of Figures~1 and~2 applied to the rest of the basis $Q$. 

  Theorem \ref{thm:main}(II) applied to the elements of $Q$ will yield
  \begin{equation} 
   \begin{array}{ll}
   \bconst 1 1 = (\aconst 2 1 2)^2 - \aconst 1 1 1 \aconst 2 1 2 ,\quad
    & \bconst 1 2 = - \aconst 1 1 2 \aconst 2 1 2 ,\\
   \bconst 2 3 = - \aconst 2 2 3 \aconst 3 2 3, \quad
    &\bconst 3 3 = (\aconst 2 2 3)^2 - \aconst 2 2 3 \aconst 3 3 3. \label{eq:exampleconditionii}
    \end{array}
  \end{equation}
  Finally, condition~(III) is vacuous from the computations above; namely, condition~(III) is equivalent to
  \[
    \aconst 1 1 1 \bconst 1 2 - \bconst 1 1 \aconst 1 1 2 - \bconst 1 2 \aconst 2 1 2 ~=~ \aconst 2 1 2 \bconst 2 3 - \bconst 1 2 \aconst 2 2 3 ~=~ \aconst 2 2 3 \bconst 2 3 + \aconst 3 2 3 \bconst 3 3 - \bconst 2 3 \aconst 3 3 3 = 0,
  \]
  and the first three expressions are equal to $0$ by 
  \eqref{eq:exampleconditionii}.
  Thus, the PBW deformations of $A$ are determined and classified by the free parameters $\aconst 1 1 1$, $\aconst 1 1 2$, $\aconst 2 1 2$, $\aconst 2 2 3$, $\aconst 3 3 3$ with all other parameters determined by  \eqref{eq:exampleconditioni} and \eqref{eq:exampleconditionii}.
\end{example}


\section{Special cases} \label{sect:cases}
 
Recall that  $V$ is the $\kk$-vector space with basis $\{x_1, \ldots, x_n\}$ and $R \subseteq T(V)$ is the span of a set of quadratic monomials in $x_i$.
In the last section, we considered maps $\alpha: R \to V$ and $\beta: R \to \kk$, corresponding filtration parameters $\aconst m i j$ and $\bconst i j$ (see Definition \ref{def:defparams}), along
with $P := \{r - \alpha(r) - \beta(r): r \in R\}$, and we provided explicit necessary and sufficient conditions on the filtration parameters so that the filtered algebra
$D = T(V) / (P)$
is a PBW deformation of the graded algebra
$A = T(V) / (R)$.
Namely, we established the three conditions (I)-(III) of Theorem~\ref{thm:main}.
In practice, one computes the basis $Q$ of $RV \cap VR$ as given in Lemma~\ref{lem:RVVRbasis}; next, one checks (I) for each element of $Q$ first, then (II) for each element of $Q$, then (III).
Here, we will investigate special cases in which condition (III) of Theorem~\ref{thm:main} need not be checked for a given element of $Q$, and in which condition (II) is partially verified.

By Lemma~\ref{lem:RVVRbasis}, any element of $Q$ can be written in one of the following forms, 
\[
  x_ix_jx_k, \quad
  x_i^2x_k, \quad
  x_ix_k^2, \quad
  x_ix_jx_i, \quad
  \text{or} \quad x_i^3
\]
where $i, j, k$ are distinct; we examine each of these cases separately.
The results achieved are Propositions \ref{prop:vac ijk}, \ref{prop:vac iik}, \ref{prop:vac ikk}, \ref{prop:vac iji}, and \ref{prop:vac iii}, respectively. 
Recall that condition (II) of Theorem~\ref{thm:main} gives restrictions on the filtration parameters of $D$ in terms of the expressions denoted by $d_r^{ijk}$ for each $r \in \{1, \ldots ,n\}$.
Our results will be presented in terms of $d_r^{ijk}$ for particular values of $r$. 
Conditions under which the propositions apply will also be represented by conditions on subgraphs of $\Gamma(A)$ (see Notation~\ref{notation:gamma(A)}).
For convenience of notation, we will sometimes write $\aconst m i j$ and $\bconst i j$ even if it is not known whether $x_i x_j \in R$, with the understanding that all such constants are zero if $x_i x_j \notin R$.

\begin{proposition} \label{prop:vac ijk} 
Assume that part~\textup{(I)} of Theorem~\ref{thm:main} holds for each element of $Q$, and take $x_i x_j x_k \in Q$ with $i, j, k$ distinct.
   \begin{enumerate}
   	\item[\textnormal{(1)}] Suppose that 
   	\begin{itemize}
   		\item $x_ix_{m'}\notin R$ for each $m'\neq  i, j$; and
   		\item $x_mx_k\notin R$ for each $m \neq  j, k$.
   	\end{itemize}
   	Then condition \textup{(II)} of Theorem~\ref{thm:main}, applied to $x_i x_j x_k$, holds for $r \neq i,j,k$; thus, one only needs to verify \textup{(II)} for $r = i,j,k$.
   
   	\item[\textnormal{(2)}] Along with the hypotheses of part \textnormal{(1)}, suppose that condition \textup{(II)} of Theorem~\ref{thm:main} holds for each element of $Q$ and that  
   	\[
   		x_i^2, x_k^2 \notin R \quad \text{or} \quad x_i^2, x_k^2 \in R. 
   	\]
   	Then condition \textup{(III)} of Theorem~\ref{thm:main}, applied to $x_ix_jx_k$, holds automatically.
   \end{enumerate}
\end{proposition}

\noindent The applicability of these conditions can be visualized by the subgraphs of $\Gamma(A)$ in Figure~3.
\begin{figure}[!h]
 { \footnotesize   \begin{tikzpicture}
	      \hspace{-.3in}
           \tikzset{vertex/.style = {shape = circle,fill = black,minimum size = 6pt,inner sep=0pt}}
           \tikzset{edge/.style = {-{Latex[length=2mm]}}}
           \tikzset{loopedge/.style = {-{Latex[length=2mm]}, in=45, out=135, loop}}
           \tikzset{every loop/.style={-{Latex[length=2mm]}, min distance=20mm,in=220,out=310,looseness=1}}

           \node[vertex, label=below:{$i$}] (i) at (0,0) {};
           \node[vertex, label=below:{$j$}] (j) at (1.5,0) {};
           \node[vertex, label=below:{$k$}] (k) at (3, 0) {};
           \node[vertex, label={$\forall \; m' \neq i ,j$}] (m') at (0.5,1) {};
           \node[vertex, label={$\forall \; m \neq j, k$}] (m) at (2.5, 1) {};
           
           \path[->] (i) edge [loop below, white] node {} ();
           \draw[edge, ] (i) to (j);
           \draw[edge, ] (j) to (k);
           \draw[edge, dashed] (i) to (m');
           \draw[edge, dashed] (m) to (k);
           
         \end{tikzpicture}
         
            \vspace{-.3in}

     \begin{tikzpicture}
           \tikzset{vertex/.style = {shape = circle,fill = black,minimum size = 6pt,inner sep=0pt}}
           \tikzset{edge/.style = {-{Latex[length=2mm]}}}
           \tikzset{loopedge/.style = {-{Latex[length=2mm]}, in=260, out=210, loop}}
           \tikzset{every loop/.style={-{Latex[length=2mm]}, min distance=10mm,in=220,out=310,looseness=1}}
         
           \node[vertex, label=below:$i$] (i) at (0,0) {};
           \node[vertex, label=below:$j$] (j) at (1.5,0) {};
           \node[vertex, label=below:$k$] (k) at (3,0) {};
           \node[vertex, label={$\forall \; m' \neq j$}] (m) at (0.5,1) {};
           \node[vertex, label={$\forall \; m \neq j$}] (m') at (2.5,1) {};
           
           \path[->] (i) edge  [loop below, dashed] node {} ();
           \draw[edge] (i) to (j);
           \draw[edge] (j) to (k);
           \path[->] (k) edge [loop below, dashed] node {} ();
           \draw[edge, dashed] (i) to (m);
           \draw[edge, dashed] (m') to (k);
         \end{tikzpicture}
         \begin{tikzpicture}
           \tikzset{vertex/.style = {shape = circle,fill = black,minimum size = 6pt,inner sep=0pt}}
           \tikzset{edge/.style = {-{Latex[length=2mm]}}}
           \tikzset{loopedge/.style = {-{Latex[length=2mm]}, in=260, out=210, loop}}
           \tikzset{every loop/.style={-{Latex[length=2mm]}, min distance=10mm,in=220,out=310,looseness=1}}
         
           \node[vertex, label=below:$i$] (i) at (0,0) {};
           \node[vertex, label=below:$j$] (j) at (1.5,0) {};
           \node[vertex, label=below:$k$] (k) at (3,0) {};
           \node[vertex, label={$\forall \; m' \neq i,j$}] (m) at (0.5,1) {};
           \node[vertex, label={$\forall \; m \neq j,k$}] (m') at (2.5,1) {};
           
           \path[->] (i) edge  [loop below] node {} ();
           \draw[edge] (i) to (j);
           \draw[edge] (j) to (k);
           \path[->] (k) edge [loop below] node {} ();
           \draw[edge, dashed] (i) to (m);
           \draw[edge, dashed] (m') to (k);
         \end{tikzpicture}}

\caption{Here, $\rightarrow$ (resp. $\dashrightarrow$) depicts an element of $R$ (resp. not in $R$).
   If $\Gamma(A)$ contains the top subgraph (resp. one of the bottom subgraphs), then Prop.~\ref{prop:vac ijk}(1) (resp. (2))   applies to $x_i x_j x_k \in Q$.}
   \end{figure}
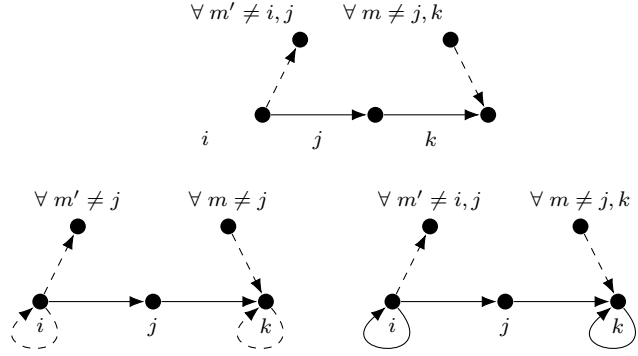

\begin{proof}[Proof of Proposition \ref{prop:vac ijk}]
  (1) By assumption,  $x_ i x_{m'}, x_m x_k \notin R$ for all $m' \neq i,j$ and $m \neq j,k$.
  Let $r \neq i,j,k$.
  Then,
  \begin{align*}
    d^{ijk}_r 
      &= \sum_{ \substack{m \in \{1, \ldots, n\} \\ x_m x_k \in R}} \aconst m i j \aconst r m k \ 
        - \sum_{\substack{m' \in \{1, \ldots, n\} \\ x_i x_{m'} \in R}} \aconst r i {m'} \aconst {m'} j k \\
      &= \aconst j i j \aconst r j k + \aconst k i j \aconst r k k - \aconst r i i \aconst i j k - \aconst r i j \aconst j j k .
   \end{align*}
  Since $x_i x_r, x_r x_k \notin R$, we get that $\aconst r j k = \aconst r i j = 0$ by Theorem~\ref{thm:main} (I.b) and (I.a), applied to $x_ix_jx_k$.
  Therefore, $d^{ijk}_r = \aconst k i j \aconst r k k - \aconst r i i \aconst i j k$.
  If $x_i^2 \in R$, then we can apply Theorem~\ref{thm:main} to $x_i^3 \in Q$; in that case, since $x_ix_r \notin R$ we have by Theorem~\ref{thm:main}(I.b)  that $\aconst r i i = 0$.
  On the other hand, if $x_i^2 \notin R$, then $\aconst r i i = 0$ by definition, so in either case, $\aconst r i i = 0$.
 We see similarly that $\aconst r k k = 0$.
   Therefore, $d_r^{i j k} = 0$, giving that condition (II), when applied to $x_ix_jx_k$, holds for $r \neq i, j, k$.
  
  (2) Now suppose also that condition (II) holds and that both $x_i^2, x_k^2 \notin R$ or both $x_i^2, x_k^2 \in R$.
  Then Theorem~\ref{thm:main}(III), applied to $x_ix_jx_k$, reduces to showing that
  \begin{equation*}
  \aconst j i j \bconst j k + \aconst k i j \bconst k k - \bconst i i \aconst i j k - \bconst i j \aconst j j k = 0.
  \end{equation*}
    Note that if $x_i^2, x_k^2 \notin R$, then $\bconst i i = \bconst k k = 0$ by definition.
    On the other hand, if $x_i^2, x_k^2 \in R$, then $x_i^2 x_j, x_j x_k^2 \in Q$.
    Since $x_i x_k \notin R$, applying (I.b) to $x_i^2 x_j$ yields $\aconst k i j = 0$.
    Similarly, applying (I.a) to $x_j x_k^2$ yields $\aconst i j k = 0$.
    In both cases we have $\aconst k i j \bconst k k = \bconst i i \aconst i j k = 0$.
  Thus, we only need to show that 
  \begin{equation} \label{4.1.3}
    \aconst j i j \bconst j k - \bconst i j \aconst j j k = 0.
  \end{equation}
  Applying (II) to $x_i x_j x_k$ for $r = i$ and $r = k$, we get $\bconst j k = \aconst j i j \aconst i j k - \aconst i i j \aconst j j k $ and $\bconst i j = \aconst k i j \aconst j j k - \aconst j i j \aconst k j k$.
 Now, if $x_i^2, x_k^2 \in R$, then as noted above, $\aconst i j k = \aconst k i j = 0$.
  On the other hand, if $x_i^2, x_k^2 \notin R$, by (I) applied to $x_i x_j x_k$, we still have $\aconst i j k = \aconst k i j = 0$.
  Thus, $\bconst i j = - \aconst j i j \aconst k j k$ and $\bconst j k = -\aconst i i j \aconst j j k$.
  Substituting for $\bconst i j$ and $\bconst j k$ in \eqref{4.1.3} yields
  \[
    - \aconst j i j \aconst i i j \aconst j j k + \aconst j i j \aconst k j k \aconst j j k = 0.
  \]
 Since $x_ix_k \notin R$, we have by (I.c), applied to $x_ix_jx_k$, that $\aconst i i j = \aconst k j k$.
 Thus, (III) holds for $x_i x_j x_k$.
\end{proof}


\begin{proposition} \label{prop:vac iik}
  Let $x_i^2 x_k \in Q$ for i, k distinct, and assume that Theorem~\ref{thm:main}\textup{(I)} holds for each element of $Q$.

  \begin{enumerate}
    \item[\textnormal{(1)}]
      Suppose that
      \begin{itemize}
      	\item $x_i x_{m'} \notin R$ for each $m' \neq  i, k$; and
      	\item $x_m x_k \notin R$ for each $m \neq  i, k$.
      \end{itemize}
    Then condition \textup{(II)} of Theorem~\ref{thm:main}, applied to $x_i^2 x_k$, holds for $r \neq i, k$; thus, one only needs to verify \textup{(II)} for $r = i,k$.

    \item[\textnormal{(2)}] 
      Along with the hypotheses of part \textnormal{(1)}, suppose that Theorem~\ref{thm:main}\textup{(II)} holds for each element of $Q$ and that
      \[
        x_k^2 \notin R \quad
          \text{or} \quad
          x_k x_i \notin R. 
      \]
      Then condition \textup{(III)} of Theorem~\ref{thm:main}, applied to $x_i^2 x_k$, holds automatically.
  \end{enumerate}
\end{proposition}

\noindent These conditions are depicted by the subgraphs of $\Gamma(A)$ in Figure~4. 

\begin{figure}[h!]
 { \footnotesize      \begin{tikzpicture}
        \tikzset{vertex/.style = {shape = circle,fill = black,minimum size = 6pt,inner sep=0pt}}
        \tikzset{edge/.style = {-{Latex[length=2mm]}}}
        \tikzset{loopedge/.style = {-{Latex[length=2mm]}, in=260, out=210, loop}}
        \tikzset{every loop/.style={-{Latex[length=2mm]}, min distance=10mm,in=220,out=310,looseness=1}}
      
        \node[vertex, label=below:$i$] (i) at (0,0){};
        \node[vertex, label=below:$k$] (k) at (3,0){};
        \node[vertex, label={$\forall \; m' \neq i, k$}] (m) at (0.5,1) {};
        \node[vertex, label={$\forall \; m \neq i, k$}] (m') at (2.5, 1) {};
        
        \draw[loopedge] (i) to (i);
        \draw[edge, bend right, looseness=.5] (i) to (k);
        \draw[edge, dashed] (i) to (m);
        \draw[edge, dashed] (m') to (k);
      \end{tikzpicture}

      \begin{tikzpicture}
        \tikzset{vertex/.style = {shape = circle,fill = black,minimum size = 6pt,inner sep=0pt}}
        \tikzset{edge/.style = {-{Latex[length=2mm]}}}
        \tikzset{loopedge/.style = {-{Latex[length=2mm]}, in=260, out=210, loop}}
        \tikzset{every loop/.style={-{Latex[length=2mm]}, min distance=10mm,in=220,out=310,looseness=1}}
      
        \node[vertex, label=below:$i$] (i) at (0,0) {};
        \node[vertex, label=below:$k$] (k) at (3,0) {};
        \node[vertex, label={$\forall \; m' \neq i, k$}] (m') at (0.5,1) {};
        \node[vertex, label={$\forall \; m \neq i$}] (m) at (2.5,1) {};
        
        \draw[loopedge] (i) to (i);
        \draw[loopedge, dashed] (k) to (k);
        \draw[edge, bend right,looseness=.5] (i) to (k); 
        \draw[edge, dashed] (i) to (m');
        \draw[edge, dashed] (m) to (k);
      \end{tikzpicture}
      \begin{tikzpicture}
        \tikzset{vertex/.style = {shape = circle,fill = black,minimum size = 6pt,inner sep=0pt}}
        \tikzset{edge/.style = {-{Latex[length=2mm]}}}
        \tikzset{loopedge/.style = {-{Latex[length=2mm]}, in=260, out=210, loop}}
        \tikzset{every loop/.style={-{Latex[length=2mm]}, min distance=10mm,in=220,out=310,looseness=1}}
      
        \node[vertex, label=below:$i$] (i) at (0,0){};
        \node[vertex, label=below:$k$] (k) at (3,0){};
        \node[vertex, label={$\forall \; m' \neq i, k$}] (m') at (0.5,1) {};
        \node[vertex, label={$\forall \; m \neq i, k$}] (m) at (2.5, 1) {};
        
        \draw[loopedge] (i) to (i);
        \draw[edge, bend right, looseness=.5] (i) to (k);
        \draw[edge, dashed] (i) to (m');
        \draw[edge, dashed] (m) to (k);
        \draw[edge, dashed, bend right, looseness=.5] (k) to (i);
      \end{tikzpicture}}
      \caption{
       Here, 
        $\rightarrow$ (resp.  $\dashrightarrow$) depicts an element in $R$ (resp. not in $R$). If $\Gamma(A)$ contains the top subgraph (resp. one of the bottom subgraphs), then Prop.~\ref{prop:vac iik}(1) (resp. (2)) applies for $x_i^2 x_k \in Q$.
      }
\end{figure}
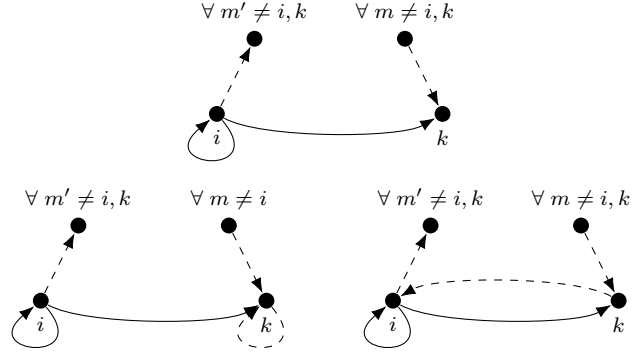

\begin{proof}[Proof of Proposition \ref{prop:vac iik}]
  (1) By assumption, $x_i x_{m'}, x_m x_k \notin R$ for all $m' \neq i,k$ and $m \neq i,k$.
  Let $r \neq i,k$.
  Then,
  \begin{align*}
    d^{i i k}_r
      &= \sum_{ \substack{m \in \{1, \ldots, n\} \\ x_m x_k \in R}} \aconst m i i \aconst r m k \ 
        - \sum_{\substack{m' \in \{1, \ldots, n\} \\ x_i x_{m'} \in R}} \aconst r i {m'} \aconst {m'} i k \\
      &= \aconst i i i \aconst r i k + \aconst k i i \aconst r k k - \aconst r i i \aconst i i k - \aconst r i k \aconst k i k.
  \end{align*}
  Since, $x_i x_r, x_r x_k \notin R$, by Theorem~\ref{thm:main}(I), $\aconst r i k = \aconst r i i = 0$.
  Thus, $d^{iik}_r = \aconst k i i \aconst r k k$.
  If $x_k^2 \in R$, then we can apply Theorem~\ref{thm:main} to $x_k^3 \in Q$; in that case, since $x_rx_k \notin R$ we have by Theorem~\ref{thm:main}(I.b)  that $\aconst r k k = 0$.  
  On the other hand, if $x_k^2 \notin R$, then $\aconst r k k= 0$ by definition, so in either case, $\aconst r k k = 0$.
  Therefore, $d^{iik}_r = 0$ so (II), when applied to $x_i^2 x_k$, holds for $r \neq i,k$.

  (2) Now suppose also that condition (II) holds and that $x_k^2 \notin R$ or $x_k x_i \notin R$.
  Then Theorem~\ref{thm:main}(III), applied to $x_i^2 x_k$, reduces to showing that
  \[
    \aconst i i i \bconst i k + \aconst k i i \bconst k k - \bconst i i \aconst i i k - \bconst i k \aconst k i k = 0.
  \]
 Since $x_k^2 \notin R$ or $x_k x_i \notin R$, we have $\aconst k i i = 0$; in the first case, this follows from Theorem~\ref{thm:main}(I.a) applied to $x_i^2 x_k$ and in the second case, this follows from Theorem~\ref{thm:main}(I.a) applied to $x_i^3$.
  Thus, we only need to show that
  \begin{equation} \label{iik3}
    (\aconst i i i - \aconst k i k) \bconst i k - \bconst i i \aconst i i k = 0.
  \end{equation}
  Applying (II) to $x_i^2 x_k$ for $r = i$ and $r = k$, we get $\bconst i k = - \aconst i i k \aconst k i k$ and $\bconst i i =  \aconst k i k (\aconst k i k - \aconst i i i )$.
  Substituting for $\bconst i i$ and $\bconst i k$ in \eqref{iik3} yields 
 (III) for $x_i^2 x_k$, as desired.
\end{proof}


\begin{proposition} \label{prop:vac ikk}
Let $x_ix_k^2\in Q$ with $i \neq  k$, and assume that Theorem~\ref{thm:main}\textup{(I)}  holds for each element of $Q$. 

\begin{enumerate}
   \item[\textnormal{(1)}]
      Suppose that
      \begin{itemize}
      	\item $x_i x_{m'} \notin R$ for each $m' \neq  i, k$; and
      	\item $x_m x_k \notin R$ for each $m \neq  i, k$
      \end{itemize}
      Then condition \textup{(II)} Theorem~\ref{thm:main}, applied to $x_ix_k^2$, holds for $r \neq i,k$; thus, one only needs to verify \textup{(II)} for $r=i,k$.

	\item[\textnormal{(2)}] 
   	Along with the hypotheses of part \textnormal{(1)}, suppose that  Theorem~\ref{thm:main}\textup{(II)} holds for each element of $Q$ and that
   	\[
         x_i^2 \notin R \quad \text{or} \quad x_kx_i \notin R.
      \] 
      Then condition \textup{(III)} of Theorem~\ref{thm:main}, applied to $x_i x_k^2$, holds automatically.
\end{enumerate}

\end{proposition}

These conditions can be visualized by the subgraphs of $\Gamma(A)$ in Figure~5.

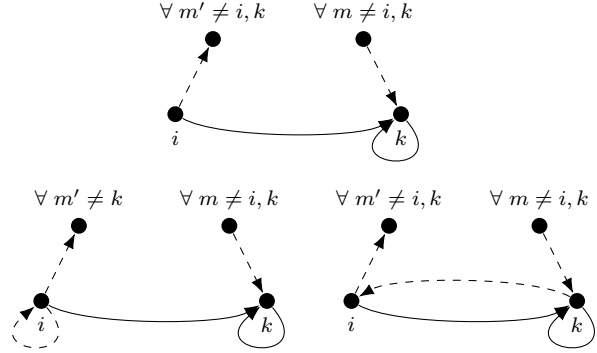
\begin{figure}[!h]
   { \footnotesize    \begin{tikzpicture}
		\tikzset{vertex/.style = {shape = circle,fill = black,minimum size = 6pt,inner sep=0pt}}
        \tikzset{edge/.style = {-{Latex[length=2mm]}}}
        \tikzset{loopedge/.style = {-{Latex[length=2mm]}, in=260, out=210, loop}}
        \tikzset{every loop/.style={-{Latex[length=2mm]}, min distance=10mm,in=220,out=310,looseness=1}}
      
        \node[vertex, label=below:$i$] (i) at (0,0) {};
        \node[vertex, label=below:$k$] (k) at (3,0) {};
        \node[vertex, label={$\forall \; m' \neq i, k$}] (m) at (0.5,1) {};
        \node[vertex, label={$\forall \; m \neq i, k$}] (m') at (2.5,1) {};
        
        \draw[edge, bend right, looseness=.5] (i) to (k);
		\draw[loopedge] (k) to (k);
        \draw[edge, dashed] (i) to (m);
        \draw[edge, dashed] (m') to (k);

      \end{tikzpicture}

      \begin{tikzpicture}
	  \tikzset{vertex/.style = {shape = circle,fill = black,minimum size = 6pt,inner sep=0pt}}
      \tikzset{edge/.style = {-{Latex[length=2mm]}}}
      \tikzset{loopedge/.style = {-{Latex[length=2mm]}, in=260, out=210, loop}}
      \tikzset{every loop/.style={-{Latex[length=2mm]}, min distance=10mm,in=220,out=310,looseness=1}}
      
        \node[vertex, label=below:$i$] (i) at (0,0) {};
        \node[vertex, label=below:$k$] (k) at (3,0) {};
        \node[vertex, label={$\forall \; m' \neq k$}] (m) at (0.5,1) {};
        \node[vertex, label={$\forall \; m \neq i, k$}] (m') at (2.5,1) {};
        
        \draw[loopedge, dashed] (i) to (i);
        \draw[edge, bend right, looseness=.5] (i) to (k);
		\draw[loopedge] (k) to (k); 
        \draw[edge, dashed] (i) to (m);
        \draw[edge, dashed] (m') to (k);
      \end{tikzpicture}
            \begin{tikzpicture}
		\tikzset{vertex/.style = {shape = circle,fill = black,minimum size = 6pt,inner sep=0pt}}
        \tikzset{edge/.style = {-{Latex[length=2mm]}}}
        \tikzset{loopedge/.style = {-{Latex[length=2mm]}, in=260, out=210, loop}}
        \tikzset{every loop/.style={-{Latex[length=2mm]}, min distance=10mm,in=220,out=310,looseness=1}}
      
        \node[vertex, label=below:$i$] (i) at (0,0) {};
        \node[vertex, label=below:$k$] (k) at (3,0) {};
        \node[vertex, label={$\forall \; m' \neq i, k$}] (m) at (0.5,1) {};
        \node[vertex, label={$\forall \; m \neq i, k$}] (m') at (2.5,1) {};
        
        \draw[edge, bend right, looseness=.5] (i) to (k);
		\draw[loopedge] (k) to (k);
        \draw[edge, dashed] (i) to (m);
        \draw[edge, dashed] (m') to (k);
        \draw[edge, bend right, dashed, looseness=.5] (k) to (i);
       \end{tikzpicture}}

   \caption{
   Here, 
        $\rightarrow$ (resp.  $\dashrightarrow$) depicts an element in $R$ (resp. not in $R$). If $\Gamma(A)$ contains the top subgraph (resp. one of the bottom subgraphs), then Prop.~\ref{prop:vac ikk}(1) (resp. (2)) applies for $x_i x_k^2 \in Q$.}
\end{figure}

\begin{proof}[Proof of Proposition \ref{prop:vac ikk}]
The proof is similar to that of Proposition~\ref{prop:vac iik}.
\end{proof}


\begin{proposition} \label{prop:vac iji}
  Let $x_i x_j x_i\in Q$ with $i \neq j$, and assume that Theorem~\ref{thm:main}\textup{(I)} holds for each element of $Q$. 
   \begin{enumerate}
   	\item[\textnormal{(1)}]
   	  Suppose that for each $m \neq  i, j$, one of the following conditions holds:
        \[
         	x_ix_m, x_mx_i \notin R
         	\quad \text{ or } \quad
         	x_jx_m, x_mx_j \notin R
         	\quad \text{ or } \quad
         	x_m^2 \notin R.
        \]
        Then condition \textup{(II)} of Theorem~\ref{thm:main}, applied to $x_i x_j x_i$, holds for $r \neq i, j$; thus, one only needs to verify \textup{(II)} for $r=i,j$.
   	\item[\textnormal{(2)}] 
   	  Along with the hypotheses of part \textnormal{(1)}, suppose that condition \textup{(II)} of Theorem~\ref{thm:main} holds for each element of $Q$ and that both $x_i^2, x_j^2 \notin R$.
   	  Then condition \textup{(III)} of Theorem~\ref{thm:main}, applied to $x_ix_jx_i$, holds automatically.
   \end{enumerate}
\end{proposition}

These conditions are pictured as subgraphs of $\Gamma(A)$ in Figure~6.

\begin{figure}[h!]
    {\footnotesize  \begin{tikzpicture}
		\tikzset{vertex/.style = {shape = circle,fill = black,minimum size = 6pt,inner sep=0pt}}
        \tikzset{edge/.style = {-{Latex[length=2mm]}}}
        \tikzset{loopedge/.style = {-{Latex[length=2mm]}, in=260, out=210, loop}}
        \tikzset{every loop/.style={-{Latex[length=2mm]}, min distance=20mm,in=220,out=310,looseness=1}}
      
        \node[vertex, label=below:$i$] (i) at (0,0) {};
        \node[vertex, label=below:$j$] (j) at (2.5,0) {};
        \node[vertex, label={$m \neq i,j$}] (m) at (0.75,2) {};
        \node[vertex, white] (w) at (3.5, 0){};
        
        \draw[edge, bend left, looseness=.5] (i) to (j);
        \draw[edge, bend left, looseness=.5] (j) to (i);
        \draw[edge, dashed, bend right, looseness=.5] (m) to (i);
        \draw[edge, dashed, bend right, looseness=.5] (i) to (m);
      \end{tikzpicture}
      \begin{tikzpicture}
       	\tikzset{vertex/.style = {shape = circle,fill = black,minimum size = 6pt,inner sep=0pt}}
        \tikzset{edge/.style = {-{Latex[length=2mm]}}}
        \tikzset{loopedge/.style = {-{Latex[length=2mm]}, in=260, out=210, loop}}
        \tikzset{every loop/.style={-{Latex[length=2mm]}, min distance=20mm,in=220,out=310,looseness=1}}
      
        \node[vertex, label=below:$i$] (i) at (0,0) {};
        \node[vertex, label=below:$j$] (j) at (2.5,0) {};
        \node[vertex, label={$m \neq i,j$}] (m) at (1.75,2) {};
        
        \draw[edge, bend right, looseness=.5] (i) to (j);
        \draw[edge, bend right, looseness=.5] (j) to (i);
        \draw[edge, dashed, bend left, looseness=.5] (m) to (j);
        \draw[edge, dashed, bend left, looseness=.5] (j) to (m);
      \end{tikzpicture}
      \begin{tikzpicture}
       	\tikzset{vertex/.style = {shape = circle,fill = black,minimum size = 6pt,inner sep=0pt}}
        \tikzset{edge/.style = {-{Latex[length=2mm]}}}
        \tikzset{loopedge/.style = {-{Latex[length=2mm]}, in=260, out=210, loop}}
        \tikzset{every loop/.style={-{Latex[length=2mm]}, min distance=10mm,in=220,out=310,looseness=1}}
      
        \node[vertex, label=below:$i$] (i) at (0,0) {};
        \node[vertex, label=below:$j$] (j) at (2.5,0) {};
        \node[vertex, label={$m \neq i,j$}] (m) at (1.25,2) {};
        \node[vertex, white] (w) at (-1, 0) {};
        
        \draw[edge, bend right, looseness=.5] (i) to (j);
        \draw[edge, bend right, looseness=.5] (j) to (i);
        \path[->] (m) edge [loop below, dashed] node {} ();
      \end{tikzpicture}

      \begin{tikzpicture}
		\tikzset{vertex/.style = {shape = circle,fill = black,minimum size = 6pt,inner sep=0pt}}
        \tikzset{edge/.style = {-{Latex[length=2mm]}}}
        \tikzset{loopedge/.style = {-{Latex[length=2mm]}, in=260, out=210, loop}}
        \tikzset{every loop/.style={-{Latex[length=2mm]}, min distance=10mm,in=220,out=310,looseness=1}}
      
        \node[vertex, label=below:$i$] (i) at (0,0) {};
        \node[vertex, label=below:$j$] (j) at (2,0) {};
        \node[vertex, label={$m \neq i,j$}] (m) at (0.75,2) {};
        \node[vertex, white] (w) at (3.5, 0){};
        
        \draw[edge, bend left, looseness=.5] (i) to (j);
        \draw[edge, bend left, looseness=.5] (j) to (i);
        \draw[edge, dashed, bend right, looseness=.5] (m) to (i);
        \draw[edge, dashed, bend right, looseness=.5] (i) to (m);
        \path[->] (i) edge [loop below, dashed] node {} ();
        \path[->] (j) edge [loop below, dashed] node {} ();
      \end{tikzpicture}
      \begin{tikzpicture}
       	\tikzset{vertex/.style = {shape = circle,fill = black,minimum size = 6pt,inner sep=0pt}}
        \tikzset{edge/.style = {-{Latex[length=2mm]}}}
        \tikzset{loopedge/.style = {-{Latex[length=2mm]}, in=260, out=210, loop}}
        \tikzset{every loop/.style={-{Latex[length=2mm]}, min distance=10mm,in=220,out=310,looseness=1}}
      
        \node[vertex, label=below:$i$] (i) at (0,0) {};
        \node[vertex, label=below:$j$] (j) at (2,0) {};
        \node[vertex, label={$m \neq i,j$}] (m) at (1.75,2) {};
        
        \draw[edge, bend right, looseness=.5] (i) to (j);
        \draw[edge, bend right, looseness=.5] (j) to (i);
        \draw[edge, dashed, bend left, looseness=.5] (m) to (j);
        \draw[edge, dashed, bend left, looseness=.5] (j) to (m);
        \path[->] (i) edge [loop below, dashed] node {} ();
        \path[->] (j) edge [loop below, dashed] node {} ();
      \end{tikzpicture}
      \begin{tikzpicture}
       	\tikzset{vertex/.style = {shape = circle,fill = black,minimum size = 6pt,inner sep=0pt}}
        \tikzset{edge/.style = {-{Latex[length=2mm]}}}
        \tikzset{loopedge/.style = {-{Latex[length=2mm]}, in=260, out=210, loop}}
        \tikzset{every loop/.style={-{Latex[length=2mm]}, min distance=10mm,in=220,out=310,looseness=1}}
      
        \node[vertex, label=below:$i$] (i) at (0,0) {};
        \node[vertex, label=below:$j$] (j) at (2,0) {};
        \node[vertex, label={$m \neq i,j$}] (m) at (1.25,2) {};
        \node[vertex, white] (w) at (-1, 0) {};
        
        \draw[edge, bend right, looseness=.5] (i) to (j);
        \draw[edge, bend right, looseness=.5] (j) to (i);
        \draw[loopedge, dashed] (m) to (m);
        \path[->] (i) edge [loop below, dashed] node {} ();
        \path[->] (j) edge [loop below, dashed] node {} ();
      \end{tikzpicture}}
    \caption{
Here, 
        $\rightarrow$ (resp.  $\dashrightarrow$) depicts an element in $R$ (resp. not in $R$). If $\Gamma(A)$ contains the top subgraph (resp. one of the bottom subgraphs), then Prop.~\ref{prop:vac iji}(1) (resp. (2)) applies for $x_i x_j x_i \in Q$.
}
\end{figure}
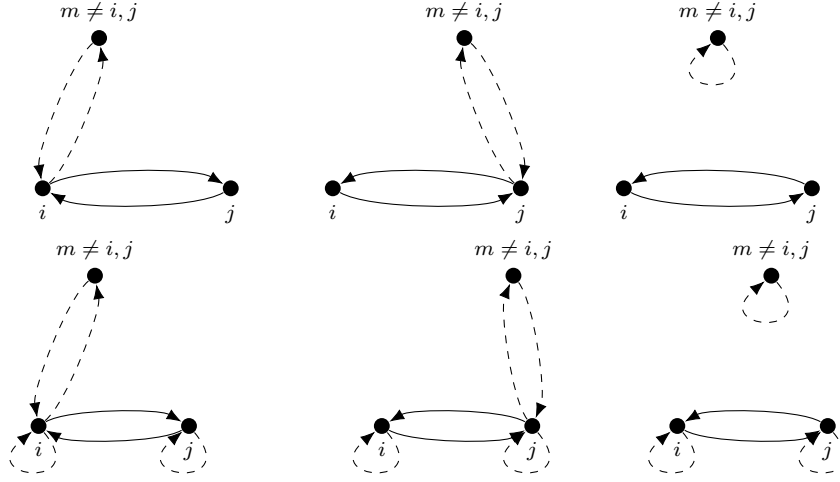

\begin{proof}   
   (1)
   Let $m \neq i,j$ and consider the following 3 cases.
   \begin{enumerate}
   	\item $x_ix_m \notin R$ and $x_mx_i \notin R$
   	\item $x_jx_m \notin R$ and $x_mx_j \notin R$
   	\item $x_i x_m \in R$ or $x_m x_i \in R$, but $x_m^2 \notin R$
   \end{enumerate}
   We will show that in any case, $\aconst m i j = \aconst m j i = \aconst m i i = 0$.
    
   In case (a), by applying Theorem~\ref{thm:main}(I) to $x_ix_jx_i$ we see that $\aconst m i j = \aconst m j i = 0$. 
   If $x_i^2 \notin R$, then by definition $\aconst m i i = 0$. Otherwise, if $x_i^2 \in R$, then we can apply Theorem~\ref{thm:main} to $x_i^3 \in Q$:  since $x_mx_i \notin R$ we have by Theorem~\ref{thm:main}(I.a)  that $\aconst m i i = 0$.   
   
   Similarly, in case (b), by applying (I) to $x_jx_ix_j$, the first two parameters listed are zero.
 If $x_i^2 \in R$, then by applying (I.b) to $x_j x_i x_i$, we get $\aconst m i i = 0$.
   
   Finally, in case (c), if $x_m x_i \in R$, then since $x_m^2 \notin R$, by applying (I.b) to $x_m x_i x_j$, we get $\aconst m i j = 0$.
   On the other hand, if $x_m x_i \notin R$, then applying (I.a) to $x_i x_j x_i$ yields $\aconst m i j = 0$.
   Similarly, regardless of whether $x_m x_j \in R$ or not, we get $\aconst m j i = 0$.
   Also, if $x_i^2 \in R$, then since $x_i x_m \in R$ or $x_m x_i \in R$, either $x_i^2 x_m \in R$ or $x_m x_i^2 \in R$.
   Therefore, by applying (I.a) to $x_i^2 x_m$ or by applying (I.b) to $x_m x_i^2$, we get $\aconst m i i = 0$.
   
   Thus, for any $m \neq i,j$, we have $\aconst m i j = \aconst m j i = \aconst m i i = 0$.  
  Therefore, for $r \neq i,j$,
  \begin{equation*}
    d^{i j i}_r
      = \sum_{ \substack{m \in \{1, \ldots, n\} \\ x_m x_i \in R}} \aconst m i j \aconst r m i \ 
        - \sum_{\substack{m' \in \{1, \ldots, n\} \\ x_i x_{m'} \in R}} \aconst r i {m'} \aconst {m'} j i 
      ~=~ \aconst j i j \aconst r j i - \aconst r i j \aconst j j i
      ~=~ 0.
  \end{equation*}
 Hence, (II) holds for $r \neq i,j$.
   
   (2)
   Now suppose also that condition (II) holds and that $x_i^2, x_j^2 \notin R$.
   Then by applying (I.c) to $x_i x_j x_i$ and $x_j x_i x_j$, we have $\aconst i i j = \aconst i j i$ and $\aconst j i j = \aconst j j i$, respectively.
   Thus, Theorem~\ref{thm:main}(III) reduces to showing that
$
     0 
     = \aconst j i j (\bconst j i - \bconst i j).
$
   Applying (II) to $x_i x_j x_i$ for $r=i$ yields that $\bconst j i - \bconst i j = \aconst j i j \aconst i j i - \aconst i i j \aconst j j i$, which is 0, by the above.
   Thus, (III) holds automatically.
\end{proof}

\begin{proposition}\label{prop:vac iii}
Let $x_i^3 \in Q$, and assume that Theorem~\ref{thm:main}\textup{(I)}  holds for each element of $Q$.
Suppose that for each $m\neq i$, one of the following conditions holds:
\[
	x_ix_m \notin R, \quad \text{or} \quad
	x_mx_i \notin R, \quad \text{or} \quad
	x_m^2 \notin R.
\]
Then, Theorem~\ref{thm:main}\textup{(II,III)}, applied to the element $x_i^3$, hold automatically.
\end{proposition}

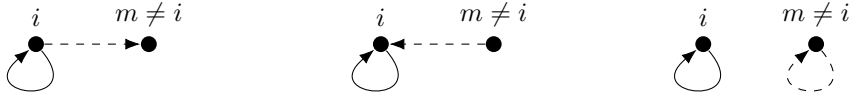
\begin{figure}[h]
      \begin{tikzpicture}
        \tikzset{vertex/.style = {shape = circle,fill = black,minimum size = 6pt,inner sep=0pt}}
        \tikzset{edge/.style = {-{Latex[length=2mm]}}}
       \tikzset{loopedge/.style = {-{Latex[length=2mm]}, in=220, out=140, loop}}
       \tikzset{every loop/.style={-{Latex[length=2mm]}, min distance=10mm,in=220,out=310,looseness=1}}
      
        \node[vertex, label= $i$] (i) at (0,0) {};
        \node[vertex, label= {$m\neq i$}] (m) at (1.5,0) {};
        \node[vertex, white](e1) at (3.5,0){};
        
        \path[->] (i) edge [loop below] node {} ();
        \draw[edge, dashed] (i) to (m);
      \end{tikzpicture}
      \begin{tikzpicture}
        \tikzset{vertex/.style = {shape = circle,fill = black,minimum size = 6pt,inner sep=0pt}}
        \tikzset{edge/.style = {-{Latex[length=2mm]}}}
        \tikzset{loopedge/.style = {-{Latex[length=2mm]}, in=220, out=140, loop}}
        \tikzset{every loop/.style={-{Latex[length=2mm]}, min distance=10mm,in=220,out=310,looseness=1}}
      
        \node[vertex, label=$i$] (i) at (0,0) {};
        \node[vertex, label= {$m \neq i$}] (m) at (1.5,0) {};
        
        \path[->] (i) edge [loop below] node {} ();
        \draw[edge, dashed] (m) to (i);
      \end{tikzpicture}
      \begin{tikzpicture}
        \tikzset{vertex/.style = {shape = circle,fill = black,minimum size = 6pt,inner sep=0pt}}
        \tikzset{edge/.style = {-{Latex[length=2mm]}}}
        \tikzset{loopedge/.style = {-{Latex[length=2mm]}, in=220, out=140, loop}}
        \tikzset{every loop/.style={-{Latex[length=2mm]}, min distance=10mm,in=220,out=310,looseness=1}}
      
        \node[vertex, white](e2) at (0,0) {};
        \node[vertex, label=  $i$] (i) at (2,0) {};
        \node[vertex, label=  {$m \neq i$}] (m) at (3.5,0) {};

        \path[->] (i) edge [loop below] node {} ();
        \path[->] (m) edge [loop below, dashed] node {} ();
      \end{tikzpicture}
      \caption{Here, 
        $\rightarrow$ (resp.  $\dashrightarrow$) depicts an element in $R$ (resp. not in $R$). If, for each $m \neq  i$, $\Gamma(A)$ contains one of the above subgraphs, then Proposition \ref{prop:vac iii} applies for $x_i^3 \in Q$.
}
\end{figure}

These conditions are pictured as subgraphs of $\Gamma(A)$ in Figure~7.

\begin{proof}
Let $m \neq  i$ and consider the following three cases:
$$ \textnormal{(a)}~ x_ix_m \notin R, \quad  \quad 
 \textnormal{(b)} ~x_mx_i \notin R \quad \quad
\textnormal{(c)}~ x_ix_m, ~x_mx_i \in R \text{ and } x_m^2 \notin R.$$

We will show that in each case, we get $\aconst m i i  = 0$. 
In case (a), $x_ix_m\notin R$, so applying Theorem \ref{thm:main}(I.b) to $x_i^3$ will yield $\aconst m i i =0$.
Similarly in case (b), $x_mx_i \notin R$, so applying (I.a) to $x_i^3$ will yield $\aconst m i i = 0$.
In case (c), we have $x_ix_m, x_mx_i \in R$ and $x_m^2 \notin R$. 
Therefore, $x_i^2x_m \in Q$. 
Since $x_m^2 \notin R$, applying (I.a) to $x_i^2 x_m$ will yield $\aconst m i i = 0$.

Therefore, for all $r$:
\[
  d^{iii}_r 
    = \sum_{ \substack{m \in \{1, \ldots, n\} \\ x_m x_i \in R}} \aconst m i i \aconst r m i \ 
          - \sum_{\substack{m' \in \{1, \ldots, n\} \\ x_i x_{m'} \in R}} \aconst r i {m'} \aconst {m'} i i 
    ~=~ \aconst i i i \aconst r i i - \aconst r i i \aconst i i i 
    ~=~ 0.
\]
 Thus Theorem (II) applied to $x_i^3$ is vacuous.

Similarly, since $\aconst m i i = 0$ in all cases, Theorem~\ref{thm:main}(III), applied to $x_i^3$, holds automatically.
\end{proof}
\section{Examples and Applications} \label{sect:examples}

The goal of this section is to illustrate Theorem~\ref{thm:main} by studying PBW deformations of several classes of quadratic monomial algebras $A = T(V)/(R)$. 
In particular, we will show how the theorem can be applied to the basis $Q$ of $RV \cap VR$ [Lemma~\ref{lem:RVVRbasis}] or, equivalently, to the graph $\Gamma(A)$ [Notation \ref{notation:gamma(A)}] in our computations.
We begin with the most trivial example: $Q = \emptyset$.

\begin{example} \label{ex:RVVR empty}
  Suppose $A = \kk \langle x_1, \dots, x_n \rangle / (R)$ such that $Q = \emptyset$ (equivalently, $RV \cap VR = 0$). 
  That is,  $\Gamma(A)$ contains no paths of length two. 
  So, Theorem~\ref{thm:main} yields no restrictions on the filtration parameters of $A$ to be deformation parameters, and thus any choice of $a_{m}^{ij}$ and $b^{ij}$, for each $x_ix_j \in R$, yields a PBW deformation of~$A$.
\end{example}

Now we study filtration parameters in the context of connected components of~$\Gamma(A)$.

\begin{lemma}\label{lem:connected components}
  Let $A = \kk \langle x_1, \dots, x_n \rangle / (R)$, for $R$ the span of some set of quadratic monomials in the $x_i$, and let $\Gamma(A)$ denote the corresponding graph as in Notation~\ref{notation:gamma(A)}.
  Let $\Gamma_1$ be a connected component of $\Gamma(A)$.
  Then Theorem~\ref{thm:main} applied to any path of length two in $\Gamma_1$ will only yield restrictions on the filtration parameters $\aconst m i j$ and $\bconst i j$ corresponding to arrows $i \to j$ in $\Gamma_1$; the filtration parameters associated to arrows in other components remain unaffected.
  
  Moreover, if every arrow $i \to j$ in $\Gamma_1$ is part of a path of length two, then for filtration parameters to yield a PBW deformation, we must have $\aconst m i j = 0$ for all vertices $m \in \Gamma \setminus \Gamma_1$.
\end{lemma}

\begin{proof}
  Let $i \to j \to k$ be a path of length two in $\Gamma_1$.
  Theorem~\ref{thm:main} applied to $x_i x_j x_k$ gives conditions only on filtration parameters corresponding to arrows with $i$, $j$, or $k$ as the source or target, which are all arrows in $\Gamma_1$.
  For the second statement, let $i \to j$ be an arrow in $\Gamma_1$.
  If we have a path of length two $i \to j \to k$ in $\Gamma$, then by Theorem~\ref{thm:main}(I.a), $\aconst m i j = 0$; indeed, there is no arrow $m \to k$ for $m \in \Gamma \setminus \Gamma_1$.
  If, on the other hand, we have $h \to i \to j$, then by Theorem~\ref{thm:main}(I.b), $\aconst m i j = 0$.
\end{proof}

\begin{remark} \label{rem:connected}
   By Lemma~\ref{lem:connected components}, in order to compute PBW deformations of a quadratic monomial algebra $A$, we may compute the deformation restrictions for each connected component of $\Gamma(A)$ separately, so long as each arrow  is part of a path of length two.
\end{remark}

\begin{example} \label{ex:loops}
  Consider the algebra $A =  \kk \langle x_1, \dots, x_n \rangle / (R)$, where $R$ only contains elements of the form $x_i^2$, i.e. $\Gamma(A)$ only contains loops. 
  By Remark~\ref{rem:connected}, in order to compute deformation parameters of $A$ it suffices to compute the deformations of the algebras $A^i := \kk\langle x_i \rangle/ (x_i^2)$ for each loop at $i$ in $\Gamma(A)$. So we proceed as such.
  
  Now take $A^i:=\kk\langle x_i \rangle/ (x_i^2)$.
  Then $\Gamma (A^i)$ is the following:
  
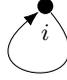
\vspace{1mm}
  
  \begin{center}
        \begin{tikzpicture}
          \tikzset{vertex/.style = {shape = circle,fill = black,minimum size = 6pt,inner sep=0pt}}
          \tikzset{edge/.style = {-{Latex[length=2mm]}}}
          \tikzset{loopedge/.style = {-{Latex[length=2mm]}, in=60, out=120, loop}}
          \tikzset{every loop/.style={-{Latex[length=2mm]}, min distance=15mm,in=220,out=310,looseness=1}}
  	\node[vertex, label=below:\small{$i$}](1) at (0,0) {};
  	
  	\path[->] (1) edge [loop below] node {} ();;
  	 \end{tikzpicture}
  \end{center}
  Observe that $Q = \{x_i^3\}$.
  We have only two filtration parameters, $\aconst iii$ and $\bconst ii$.
  For any choice of these, condition (I) of Theorem~\ref{thm:main} is vacuously true, and by Proposition~\ref{prop:vac iii}, conditions (II) and (III) of Theorem~\ref{thm:main} hold automatically.
  
  Hence, any choice of $\aconst m i i, \bconst i i \in \kk$ with $\aconst {m \neq i} i i = 0$ yields a PBW deformation of $A$.
  (Recall that the scalars $\aconst m i j$ and $\bconst i j$ for $i\neq j$ do not play a role in such a computation since $x_ix_j \notin R$.)
\end{example}

\begin{example}\label{ex:path}

Let $A = \kk \langle x_1, \dots, x_n \rangle / (x_{p_1}x_{p_2},~ x_{p_2}x_{p_3}, \dots , ~x_{p_{t-1}}x_{p_t})$, with $x_{p_i}$ all distinct and $n \geq t \geq 3$. 
Reorder these indices so that $p_d = d$ for $1\leq d \leq t$. 
In light of Remark~\ref{rem:connected}, in order to compute deformation parameters of $A$, it suffices to compute the deformations of $\widetilde A := \kk \langle x_1, \ldots, x_t \rangle / (x_1 x_2,~ x_2 x_3, ~\ldots, ~x_{t-1} x_t)$.
So we proceed as such.
 
With $\widetilde A$ as above, $\Gamma(\widetilde A)$ can be visualized as a path of length $t-1$: 

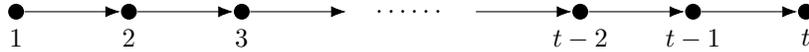
\begin{figure}[h!]
\begin{center}
      \begin{tikzpicture}
        \tikzset{vertex/.style = {shape = circle,fill = black,minimum size = 6pt,inner sep=0pt}}
        \tikzset{edge/.style = {-{Latex[length=2mm]}}}
        \tikzset{loopedge/.style = {-{Latex[length=2mm]}, in=45, out=135, loop}}

	\node[vertex, label=below:$1$](1) at (-4.5,1) {};
	\node[vertex, label=below:$2$](2) at (-3,1) {};
	\node[vertex, label=below:$3$](3) at (-1.5,1) {};
	\node[vertex, white] (4) at (0,1) {};
	\node[vertex, white] (5) at (1.5,1) {};
	\node[vertex, label=below:$t-2$](t-2) at (3,1) {};
	\node[vertex, label=below:$t-1$](t-1) at (4.5,1) {};
	\node[vertex, label=below:$t$](t) at (6,1) {};

	\draw[edge] (1) to (2);
	\draw[edge] (2) to (3);
	\draw[edge] (3) to (4);
	\draw[edge] (5) to (t-2);
	\draw[edge] (t-2) to (t-1);
	\draw[edge] (t-1) to (t);
	\path (4) -- node[auto=false]{\dots \ldots} (5);

         \end{tikzpicture}
\end{center}
\end{figure}
\noindent We will apply the conditions of Theorem~\ref{thm:main} for an element $x_i x_j x_k$ in $Q$, that is, $j = i+1$ and $k=i+2$. 
Note that $i$, $j$, and $k$ are distinct, so for each $m \neq i,j$, we get $x_m x_k \not \in R$. Also, for $m' \neq j,k$, we get $x_i x_{m'} \not \in R$. 
Therefore, Theorem~\ref{thm:main}(I) gives us that
\begin{equation} \label{eq:str}
  \aconst mij = \aconst {m'}jk = 0, \text{ for } m \neq i,j \text{ and } m' \neq j, k.
\end{equation}
Furthermore, since $x_i x_k \not \in R$, we also obtain  that
\begin{equation} \label{eq:pathIc}
  \aconst iij = \aconst kjk .
\end{equation}
Assuming the above restrictions, Theorem~\ref{thm:main}(II) only needs to be checked for $r=i,j,k$, due to Proposition~\ref{prop:vac ijk}.
If $r=i$, condition (II) gives that $\bconst jk = \aconst j i j \aconst i j k - \aconst i i j \aconst j j k.$
But by \eqref{eq:str} we get $\aconst i j k = 0$. So, $\bconst j k = -\aconst j j k \aconst i i j$. 
By a similar argument for the case $r = k$, we obtain that $\bconst i j = -\aconst j i j \aconst k j k$.
By \eqref{eq:pathIc}, these give the following restriction:
\begin{equation} \label{eq:bstr}
  \bconst ij = -\aconst i i j \aconst  j i j \text{ for all $i \to j$}.
\end{equation}
When $r=j$, we obtain the vacuous condition $\aconst j i j \aconst  j j k - \aconst j j k \aconst j i j =0$, yielding no new restrictions. 
Lastly, observe that $x_i^2, x_k^2 \not \in R$, so Proposition~\ref{prop:vac ijk} gives that Theorem~\ref{thm:main}(III) is satisfied automatically. 
Therefore, the set of restrictions on the filtration parameters given by the relations $x_ix_j, x_jx_k \in R$ with $k = j+1=i+2$ to be deformation parameters are given by \eqref{eq:str}, \eqref{eq:pathIc}, and \eqref{eq:bstr}.
\end{example}

\begin{example}\label{ex:cycle}
We now suppose that 
\begin{gather*}
 A = \kk \langle x_1, \dots, x_n \rangle / (x_{p_1}x_{p_2},~x_{p_2}x_{p_3},~\dots,~x_{p_{t-1}}x_{p_t},~x_{p_t}x_{p_1}),
\end{gather*}
in which the $x_{p_i}$ are all distinct, with $n \geq t \geq 3$.  
As in the previous example, we will reorder the $p_i$ as $1,2,\dots,t$, and it will suffice to compute the deformations of $\widetilde A := \kk \langle x_1, \ldots, x_t \rangle / (x_1, x_2, \ldots, x_{t-1} x_t, x_t x_1)$. 
The graph $\Gamma (\widetilde A)$ will be a cycle with $t$ nodes:

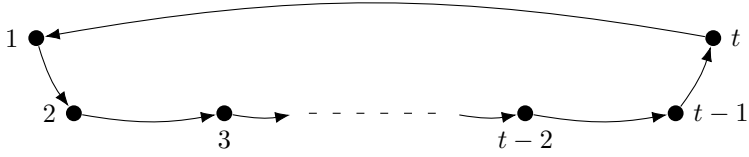
\vspace{.1in}
\begin{center}
      \begin{tikzpicture}
        \tikzset{vertex/.style = {shape = circle,fill = black,minimum size = 6pt,inner sep=0pt}}
        \tikzset{edge/.style = {-{Latex[length=2mm]}}}
        \tikzset{loopedge/.style = {-{Latex[length=2mm]}, in=45, out=135, loop}}

	\node[vertex, label=left:$1$](1) at (-4.5,1) {};
	\node[vertex, label=left:$2$](2) at (-4,.0) {};
	\node[vertex, label=below:$3$](3) at (-2,0) {};
	\node[vertex, label=below:$t-2$](t-2) at (2,0) {};
	\node[vertex, label=right:$t-1$](t-1) at (4,0) {};
	\node[vertex, label=right:$t$](t) at (4.5,1) {};
	\node[](e1) at (-1,0) {};
	\node[](e2) at (1,0) {};

	\draw[edge, bend right=10] (1) to (2);
	\draw[edge, bend right=10] (2) to (3);
	\draw[edge, bend right=10] (t-2) to (t-1);
	\draw[edge, bend right=10] (t-1) to (t);
	\draw[edge, bend right=10](t) to (1);
	\draw[edge, bend right=10](3) to (e1);
	\draw[loosely dashed] (e1) to (e2);
	\draw[edge, bend right=10] (e2) to (t-2);
         \end{tikzpicture}
\end{center}
This may be viewed as an extension of Example \ref{ex:path}, since we are appending one additional arrow from the node $p_t$ to the node $p_1$.
All of the conditions previously satisfied in Theorem~\ref{thm:main} and Proposition \ref{prop:vac ijk} are again satisfied in this example.
Proposition~\ref{prop:vac ijk} yields restrictions \eqref{eq:str}, \eqref{eq:pathIc}, \eqref{eq:bstr} on the filtration parameters for $x_ix_j, x_j x_k \in R$ to be deformation parameters.
The element $x_t x_1 \in R$ adds only the extra restriction: $\aconst {t - 1} {t - 1,} t = \aconst 1 t 1$ and $ \aconst t t 1 = \aconst 2 1 2$.
\end{example}

We will now provide an example in which the results in Section 4 do not apply to an element of $Q$.
\begin{example}\label{ex:il}
Let $A = \kk \langle x_1, x_2, x_3, x_4 \rangle / (x_1x_2,~ x_2x_3,~ x_1x_4)$, so that $\Gamma(A)$  is
 
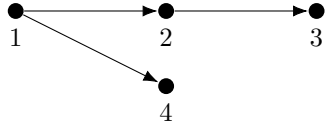
\begin{figure}[h!]
    \begin{center}
      \begin{tikzpicture}
        \tikzset{vertex/.style = {shape = circle,fill = black,minimum size = 6pt,inner sep=0pt}}
        \tikzset{edge/.style = {-{Latex[length=2mm]}}}
        \tikzset{loopedge/.style = {-{Latex[length=2mm]}, in=45, out=135, loop}}
        \tikzset{every loop/.style={-{Latex[length=2mm]}, min distance=20mm,in=220,out=310,looseness=1}}
      
        \node[vertex, label=below:$1$] (a) at (0,2) {};
        \node[vertex, label=below:$2$] (b) at (2,2) {};
        \node[vertex, label=below:$3$] (c) at (4,2) {};
        \node[vertex, label=below:$4$] (d) at (2,1) {};
        
        \draw[edge] (a) to (b);
        \draw[edge] (b) to (c);
        \draw[edge] (a) to (d);
      \end{tikzpicture}
    \end{center}
  \end{figure}
  
  \vspace{-.1in}

\noindent We have that $x_1 x_2 x_3$ is the only element of $Q$, so the conditions of Theorem~\ref{thm:main} need only be checked for this element. 
Since $x_m x_3 \not \in R$ for $m=3,4$, Theorem~\ref{thm:main}(I.a) gives that $\aconst 3 1 2 = \aconst 4 1 2 = 0$. 
Since $x_1^2 \not\in R$, (I.b) gives that $\aconst 1 2 3 = 0$. 
Lastly, since $x_1 x_3 \not \in R$, (I.c) gives that $\aconst 1 1 2 = \aconst 3 2 3$. 

For (II), we note that 
$d_r^{123} = \aconst 2 1 2 \aconst r 2 3 - \aconst r 1 2 \aconst 2 2 3 - \aconst r 1 4 \aconst 4 2 3$.
Thus, for $r=1$, we get $\bconst 2 3 = \aconst 2 1 2 \aconst 1 2 3 - \aconst 1 1 2 \aconst2 2 3 - \aconst 1 1 4 \aconst 4 2 3$.
Now since $\aconst 1 2 3 = 0$ by the preceding, we conclude that 
\begin{equation} \label{eq:r=1}
\bconst 2 3 = -\aconst 1 1 2 \aconst2 2 3 - \aconst 1 1 4 \aconst 4 2 3.
\end{equation}
In a similar fashion, setting $r=2$ gives that 
\begin{equation} \label{eq:r=2}
  -\aconst 2 1 4 \aconst 4 2 3 = 0.
\end{equation}
For $r=3$, we obtain $-\bconst 1 2 = \aconst 2 1 2 \aconst 3 2 3 - \aconst 3 1 2 \aconst 2 2 3 -\aconst 3 1 4 \aconst 4 2 3$,
and since $\aconst 3 1 2 = 0$, this gives 
\begin{equation} \label{eq:r=3}
\bconst 1 2 = -\aconst 2 1 2 \aconst 3 2 3 +\aconst 3 1 4 \aconst 4 2 3.
\end{equation}
Lastly, when $r=4$, we get
\begin{equation} \label{eq:r=4}
 \aconst 2 1 2 \aconst 4 2 3 -\aconst 4 1 4 \aconst 4 2 3 = 0.
\end{equation}

Now we need only check condition (III) of Theorem~\ref{thm:main}, which is equivalent to
$\aconst 2 1 2 \bconst 2 3 - \aconst 2 2 3 \bconst 12 - \aconst 4 2 3 \bconst 1 4 = 0.$
By \eqref{eq:r=1} and \eqref{eq:r=3}, along with $\aconst 1 1 2 = \aconst 3 2 3$ from above, we get that
\begin{equation}\label{eq:bcond}
\aconst 2 1 2 \aconst 1 1 4 \aconst 4 2 3 + \aconst 2 2 3 \aconst 3 1 4 \aconst 4 2 3 + \aconst 4 2 3 \bconst 1 4 = 0.
\end{equation}

Finally, we conclude that all PBW deformations of $A$ are of the form
\begin{gather*}
D =\kk \langle x_1, x_2, x_3, x_4 \rangle / (P),
\end{gather*}
where
\begin{gather*}
P = \kspan
\left(
\begin{array}{c|c}
\begin{array}{l}
x_1x_2 - \aconst 1 1 2 x_1 - \aconst 2 1 2 x_2 + \aconst 1 1 2 \aconst 2 1 2  -\aconst 3 1 4 \aconst 4 2 3, 
\\ 
x_2x_3  - \aconst 2 2 3 x_2 - \aconst 1 1 2 x_3 -\aconst 4 2 3 x_4  + \aconst 1 1 2 \aconst 2 2 3 + \aconst 1 1 4 \aconst 4 2 3,\\
x_1x_4 - \aconst 1 1 4 x_1 - \aconst 2 1 4 x_2 - \aconst 3  1 4 x_3 - \aconst 4 1 4 x_4 - \bconst 1 4
\end{array} &
\begin{array}{l}
\eqref{eq:r=2},\\ \eqref{eq:r=4},\\
\eqref{eq:bcond}
\end{array}
\end{array}
\right).
\end{gather*}
\end{example}


\section{Existence of Nontrivial Deformations} \label{sect:nontrivialdefs}

The purpose of this section is to show that each quadratic monomial algebra admits a nontrivial PBW deformation; we establish this result by using  Theorem~\ref{thm:main}. Let us recall some notation from Section~\ref{sect:main}. Take $A := \kk \langle x_1, \dots, x_n \rangle/(R)$ for $R$ the $\kk$-span of a collection of quadratic monomials in the $x_i$. Let $D := \kk \langle x_1, \dots, x_n \rangle/(P)$ for $P = \{x_ix_j - \sum_{m=1}^n a^{ij}_m x_m - b^{ij}~|~x_ix_j \in R\}$ with $\{a^{ij}_m\}$ and $\{b^{ij}\}$ a collection of scalars in $\kk$ called {\it filtration parameters}. Such filtration parameters are referred to as {\it deformation parameters} in the case when $D$ is a PBW deformation of $A$. Denote by $V$ the $\kk$-span of the $x_i$, and let $Q$ denote the basis of the space $RV \cap VR$.

\begin{theorem} \label{thm:nontrivialdefs}
Retain the notation above. Then any quadratic monomial algebra $A$ admits a nontrivial PBW deformation  $D$.
\end{theorem}

  \begin{proof}
    We proceed by cases:
    In each case, we exhibit filtration parameters $\{ \aconst m i j \}$ and $\{ \bconst i j \}$, not all zero, and show they satisfy the conditions of Theorem~\ref{thm:main} to be deformation parameters of $A$.
    
\vspace{.2in} 

($\star$) 

\vspace{-.3in} 

\begin{quote} 
We assume that $\aconst m i j = 0$ for all $m \neq i,j$ so that conditions (I.a), (I.b) of Theorem~\ref{thm:main} are satisfied.
\end{quote}
 
 \medskip 
 \noindent Next, the following reduction will be of use.   

\vspace{.2in} 

($\star \star$)

\vspace{-.3in} 

\begin{quote}
Note that if all $\bconst i j = 0$, then both Theorem~\ref{thm:main}(III) is satisfied and condition~(II) of Theorem~\ref{thm:main} is reduced to showing that 
    $$d^{i j k}_r :=  \sum_{ \substack{m \in \{1, \ldots, n\} \\ x_m x_k \in R}} \aconst m i j \aconst r m k \ - \sum_{\substack{m' \in \{1, \ldots, n\} \\ x_i x_{m'} \in R}} \aconst r i {m'} \aconst {m'} j k $$
    equals 0 for all $r = 1, \dots, n$.
 \end{quote}

    \medskip
    
   \noindent {\bf \underline{Case 1}}:
    Suppose for some $\ell \in \{1, 2, \ldots, n \}$ that $x_\ell^2 \in R$.
    Without loss of generality, we reindex so that $\ell = 1$.
    Then $x_1^3 \in Q$.
Set all filtration parameters of $A$ to be zero, except $\aconst 1 1 1 = 1$.

 Fix $i$, $j$, $k$ so that the monomial $x_i x_j x_k$ lies in $Q$.
    By ($\star$) and ($\star \star$), it suffices to verify condition (I.c)  of Theorem~\ref{thm:main} and that $d^{i j k}_r = 0$ for all $r$.
    
    For condition (I.c), both $\aconst i i j$ and $\aconst k j k$ are $0$, unless $i = j = 1$ or $j = k = 1$.
    In either of those cases, $x_i x_k \in R$, so condition (I.c) holds vacuously.
    
    On the other hand, in both the sums defining $d^{i j k}_r$, the only way any term is nonzero is if $i = j = k = r = 1$.
    In this case, we get $d^{1 1 1}_1 = ({\aconst 1 1 1})^2 - ({\aconst 1 1 1})^2 = 0$.
    Thus, $d^{i j k}_r = 0$ for all $r$, as required.

          \medskip
       
    \noindent {\bf \underline{Case 2}}:
    Suppose $x_\ell ^2 \not \in R$ for any $\ell \in \{1,2, \ldots, n\}$, but that $x_s x_t, x_t x_s \in R$ for some $s, t \in \{1,2,\ldots,n\}$ with $s \neq t$.
   We take $s = 1$ and $t = 2$ by reindexing.
    For any $u \in \{1, 2, \ldots, n\}$, set
    \begin{gather*}
     \mc L_{u} := \{i \in \{1,2, \ldots, n\} : x_i x_{u} \in R \} \quad \text{and} \quad \mc R_{u} := \{k \in \{1,2, \ldots, n\} : x_{u} x_k \in R \}.
    \end{gather*}
    So $x_i x_{u} x_k \in Q$ if and only if $i \in \mc L_{u}$ and $k \in \mc R_{u}$.   Let all filtrations parameters of $A$ be zero except the following:
    \begin{itemize}
      \item
        if $p \in \mc L_1$, set $\aconst p p 1 = 1$;
      \item
        if $q \in \mc R_1$, set $\aconst q 1 q = 1$;
      \item
        if $p \in \mc L_2$, set $\aconst p p 2 = 1$;
      \item
        if $q \in \mc R_2$, set $\aconst q 2 q = 1$; and
      \item
        set $\bconst 1 2 = \bconst 2 1 = -1$.
    \end{itemize}
    We will show that these are deformation parameters of $A$.
    
    Again,  fix $i$, $j$, $k$ so that the monomial $x_i x_j x_k$ lies in $Q$;
    it suffices to verify conditions (I.c), (II), (III)  of Theorem~\ref{thm:main}, due to  ($\star$).

    For condition~(I.c), suppose that $x_i x_k \notin R$.
    Then we have $\aconst i i j = \aconst k j k = 0$ unless $j \in \{1,2\}$.
    However, if $j \in \{1,2\}$, then $\aconst i i j = \aconst k j k = 1$, as required for  condition~(I.c).
    
    To verify condition~(II) of  Theorem~\ref{thm:main}, we consider whether $i$, $j$, and $k$ are in the set $\{1,2\}$.
    \medskip
    
   \noindent \underline{For (II): None of the indices $i$, $j$, $k$ are in  $\{1,2\}$.} ~ Then $d^{i j k}_r = 0$ since all $\aconst m i j$ and $\aconst {m'} j k$ are zero.
    Since $\bconst i j = \bconst j k = 0$ in this case, condition~(II) is satisfied; see ($\star \star$).
    
    \medskip
    
   \noindent \underline{For (II): Only one of the indices $i$, $j$, $k$ is in  $\{1,2\}$.} ~
Say $i \in \{1,2\}$, but $j,k \not \in \{1,2\}$.  Then $\aconst m i j = 0$ unless $m = j$. Also, $\aconst {m'} j k = 0$ for any $m'$.
    Thus, $d^{i j k}_r = \aconst j i j \aconst r j k = 0$.
    Since $\bconst i j = \bconst j k = 0$, condition~(II) is satisfied;  see ($\star \star$).
    
 If  $k \in \{1,2\}$, but $i,j \not \in \{1,2\}$, then condition (II) is verified by an argument symmetric to that above.
     
    Say $j \in \{1,2\}$, but $i,k \not \in \{1,2\}$. Then $\aconst m i j = 0$ unless $m = i$. Also, $\aconst {m'} j k = 0$ unless $m' = k$. 
Now if $i \in \mc L_k$, then $d^{i j k}_r = \aconst i i j \aconst r i k - \aconst r i k \aconst k j k$, which is 0; this holds as  $\aconst r i k = 0$ for all $r$.
 On the other hand, if $i \not \in \mc L_k$, then $d^{i j k}_r = 0$; this holds as  $\aconst m i j = 0$ for any $m \in \mc L_k$ and as $\aconst {m'} j k = 0$ for any $m' \in \mc R_i$. 
Therefore, condition~(II) is satisfied since $\bconst i j = \bconst j k = 0$; see ($\star \star$).
    
      \medskip
    
   \noindent \underline{For (II): Two of the indices $i$, $j$, $k$ are in  $\{1,2\}$.}
~ Now suppose that $i = 1$, $j = 2$, and $k \not \in \{1,2\}$. Then $\aconst m 1 2 \aconst r m k = 0$ unless $m \in \{1,2\}$ and $r = k$.
    Similarly, $\aconst r 1 {m'} \aconst {m'} 2 k = 0$ unless $r = m' = k$.
    Thus, for $r \neq k$, we get $d^{1 2 k}_r  = 0$.
    Now, if $k \in \mc R_1$, we have $d^{1 2 k}_k =  \aconst 1 1 2 \aconst k 1 k + \aconst 2 1 2 \aconst k 2 k - \aconst k 1 k \aconst k 2 k = 1$.
    If $k \not \in \mc R_1$, then we still have $d^{1 2 k}_k  = \aconst 2 1 2 \aconst k 2 k = 1$.
    Since $\bconst 1 2  = -1$ and $\bconst 2 k = 0$, condition~(II) holds.
    
   Now suppose that $i = 2$, $j = 1$, and $k \not \in \{1,2\}$; then condition~(II) is verified by an argument symmetric to that above. 
   
   Likewise,  condition (II) holds for the cases when $j,k \in \{1,2\}$ with $j \neq k$ and $i \notin \{1,2\}$ by an argument symmetric to that above.

Finally suppose that $i, k \in \{1,2\}$.
    Then $j \notin \{1,2\}$, so $\aconst m i j \aconst r m k = 0$ unless ${j = m = r}$.
    Similarly, $\aconst r i {m'} \aconst {m'} j k = 0$ unless $j = m' = r$.
    Therefore, for $r \neq j$, we get $d^{i j k}_r = 0$, and on the other hand, $d^{i j k}_j = \aconst j i j \aconst j j k - \aconst j i j \aconst j j k = 0$.
    Since $\bconst i j = \bconst j k = 0$, condition (II) is satisfied.

     \medskip
    
     \noindent \underline{For (II): All of the indices $i$, $j$, $k$ are in  $\{1,2\}$.}~ Here, we have that $(i,j,k)$ = (1,2,1) or (2,1,2) by the assumption that $x_i^2 \not \in R$.
    In order for $\aconst m i j \aconst r m k$ to be nonzero, we must have $m = j$ and $r \in \{1,2\}$.
    Similarly, for $\aconst r i {m'} \aconst {m'} j k$ to be nonzero, we must have $m' = j$ and $r \in \{1,2\}$.
    Therefore, for $r \notin \{1,2\}$, we get $d^{i j k}_r = 0$.
    On the other hand, for $r \in \{1,2\}$, we have $d^{i j k}_r = \aconst j i j \aconst r j k - \aconst r i j \aconst j j k = 0$.
    Since $\bconst i j = \bconst j i$, condition~(II) is satisfied.
	
	  \medskip
    Thus in Case 2, Theorem \ref{thm:main}(II) holds for our choice of filtration parameters.
    \medskip
        
    For condition~(III) of Theorem~\ref{thm:main}, recall that $\bconst i j = 0$ except for $\bconst 1 2 = \bconst 2 1 = -1$.
    Thus, condition~(III) is satisfied trivially if $i,k \not \in \{1,2\}$.
    
    If $i = 1$, but $k \not \in \{1,2\}$, then the sum which must be zero reduces to $-(\bconst 1 2 \aconst 2 j k)$; this is indeed zero because $\aconst 2 j k = 0$.
    Similarly, if $i = 2$, but $k \not \in \{1,2\}$, condition~(III) is satisfied. The same holds if $k \in \{1,2\}$, but $i \notin  \{1,2\}$.
    
    If $i = k = 1$, then the sum in (III) reduces to $\aconst 2 1 j \bconst 2 1 - \bconst 1 2 \aconst 2 j 1$.
    This sum is zero if $j \neq 2$, because $\aconst 2 1 j = \aconst 2 j 1 = 0$ in that case.
    However, if $j = 2$, the sum is also zero by our choice of  filtration parameters.
    Similarly, if $i = k = 2$, then (III) is satisfied.
    
    Finally, if $i = 1$ and $k = 2$, the sum reduces to $\aconst 1 1 j \bconst 1 2 - \bconst 1 2 \aconst 2 j 2$, which is zero because $\aconst 1 1 j = \aconst 2 j 2 = 0$, and similarly if $i = 2$ and $k = 1$. 
    
    \medskip
    Therefore, Theorem \ref{thm:main} is satisfied and our choice of filtration parameters yield a nontrivial PBW deformation of $A$ in this case.

       \medskip

      \noindent {\bf \underline{Case 3}}:
    Suppose neither of Cases 1 or 2 holds.
    Then with $\mc L_{u}$ and $\mc R_{u}$ as defined in Case~2, we obtain for each ${u} \in \{1,2, \ldots, n\}$ that the sets $\{{u}\}$, $\mc L_{u}$, and $\mc R_{u}$ are disjoint.
    We can choose some ${u}$ so that at least one of $\mc L_{u}$ or $R_{u}$ is nonempty, and 
    without loss of generality, we can assume that such ${u}$ is equal to 1 by reindexing.
    
 Let all filtration parameters of $A$ be zero except the following: 
    \begin{itemize}
      \item
        if $p \in \mc L_1$, set $\aconst p p 1 = 1$,
      \item
        if $q \in \mc R_1$, set $\aconst q 1 q = 1$.
    \end{itemize}
    
    If $Q$ is empty, then there is nothing to check, so any choice of filtration parameters would yield a deformation.
    (See Example~\ref{ex:RVVR empty}.)
    Otherwise, fix $i$, $j$, $k$ so that the monomial $x_i x_j x_k$ lies in $Q$. By ($\star$) and ($\star \star$), we only need to check that condition~(I.c) of Theorem~\ref{thm:main} holds, and that $d^{i j k}_r = 0$ for all $r$. 
   
   If none of $i$, $j$, or $k$ are equal to 1, the requirements above are satisfied because all parameters involved are set to zero.

    Suppose that $i = 1$.
    Then $j \in \mc R_1$ and $k \neq 1$, so condition (I.c) holds because $\aconst 1 1 j = \aconst k j k = 0$.
    Also, for any $m, m',r$, the products $\aconst m 1 j \aconst r m k$ and $\aconst {m'} j k \aconst r 1 {m'}$ are 0:
    For the first, $\aconst m 1 j = 0$ unless $m = j$, but in that case, $\aconst r m k = 0$ since neither $m (= j)$ nor $k$ is 1.
    For the second, $\aconst {m'} j k = 0$ since neither $j$ nor $k$ is 1.
    Therefore, $d_r^{1 j k} = 0$ for any $r$.
    
    The case when $k = 1$ follows in a symmetric fashion to the $i=1$ case.
    
    Now suppose that $j=1$.
   Here, we get $i \in \mc L_1$ and $k \in \mc R_1$, and that $i,k \neq 1$.
    By definition, $\aconst i i 1 = \aconst k 1 k = 1$, so condition~(I.c) is satisfied. Towards showing that $d^{i 1 k}_r = 0$ for all $r$, we have $\aconst m i 1 \aconst r m k = 0$ unless $m = i$, in which case $\aconst r m k = 0$. Thus, $\aconst m i 1 \aconst r m k = 0$.
    Similarly, $\aconst r i {m'} \aconst {m'} 1 k = 0$ unless $m' = k$, in which case $\aconst r i {m'} = 0$. So, $\aconst r i {m'} \aconst {m'} 1 k = 0$.
    Therefore, $d_r^{i 1 k} = 0 $ for all $r$, as desired.
  \end{proof}

\section*{Acknowledgments}

The authors thank the anonymous referee for their careful remarks which improved the exposition of this manuscript.
The authors are partially supported by the third author's grants from the National Science Foundation grant (\#DMS-1663775) and the Alfred P. Sloan foundation (research fellowship). Estornell and Wynne were also partially supported by Temple University's Undergraduate Research Program (URP).


\bibliography{PBW-quadratic-monomial}

\end{document}